\documentclass[a4paper,10pt,final ]{amsart}
\usepackage[utf8x]{inputenc}
\usepackage{fouriernc, amsmath, amsthm, amsfonts, amssymb,epsfig,enumerate, wrapfig, scalerel, tikz,  color, accents, fancybox, rotating, comment, afterpage, changepage, colortbl,  pdfpages, datetime, fancyhdr, float, xlop, blkarray, stmaryrd}
\usepackage{hyperref}
\hypersetup{
    pdftoolbar=true,        
    pdfmenubar=true,        
    pdffitwindow=false,     
    pdfstartview={FitH},    
    pdftitle={State sums for some super quantum link invariants},    
    pdfauthor={Louis-Hadrien Robert and Emmanuel Wagner},     
    pdfsubject={State sums for some super quantum link invariants},   
    pdfcreator={Louis-Hadrien Robert and Emmanuel Wagner},   
    pdfproducer={Louis-Hadrien Robert and Emmanuel Wagner}, 
    pdfkeywords={}, 
    pdfnewwindow=true,      
    colorlinks=true,       
    linkcolor=red,          
    citecolor=teal,        
    filecolor=magenta,      
    urlcolor=violet,          
    linkbordercolor=red,
    citebordercolor=teal,
    urlbordercolor=violet,  
    linktocpage=true
    }

\usepackage{setspace}
\usepackage{showkeys}
\usetikzlibrary{arrows}
\usetikzlibrary{decorations.markings}
\usetikzlibrary{decorations}
\usetikzlibrary{patterns}
\usetikzlibrary{positioning}
\usetikzlibrary{decorations.pathreplacing, decorations.pathmorphing}
\usepackage{tikz-3dplot}
\newcounter {res}[section]
\numberwithin{res}{section}
\newtheorem{thm}[res]{Theorem}

\newtheorem{lem}[res]{Lemma}

\newtheorem{prop}[res]{Proposition}
\newtheorem{cor}[res]{Corollary}

\theoremstyle{definition}
\newtheorem{notation}[res]{Notation}
\newtheorem{dfn}[res]{Definition}
\newtheorem{rmk}[res]{Remark}

\newcommand{\NB}[1]{\ensuremath{\vcenter{\hbox{#1}}}}

\newcommand{\ZZ}{\ensuremath{\mathbb{Z}}} 
\newcommand{\CC}{\ensuremath{\mathbb{C}}} 

\newcommand{\RR}{\ensuremath{\mathbb{R}}}

\newcommand{\id}{\mathrm{Id}}

\newcommand{\sdim}{\mathop{\mathrm{sdim}}\nolimits}
\newcommand{\gll}{\ensuremath{\mathfrak{gl}}}
\newcommand{\sll}{\ensuremath{\mathfrak{sl}}}

\newcommand{\ie}{i.~e.{} }

\newcommand{\resp}{resp.{} }

\renewcommand{\deg}[2][{}]{\ensuremath{\mathrm{deg}_{#1}(#2)}}

\newcommand\kup[1]{\left\langle #1 \right\rangle}
\newcommand\kups[1]{\left\llangle #1 \right\rrangle}

\newcommand{\svect}{\ensuremath{\mathsf{Svect}}}
\newcommand{\vect}{\ensuremath{\mathsf{vect}}}

\newcommand{\qbin}[2]{\ensuremath
\begin{bmatrix}
  #1 + #2 \\
#1 \quad #2
\end{bmatrix}
}

\newcommand{\Xing}{\NB{\scalebox{1.3}{\ensuremath{\times}}}}

\newcommand{\set}[1]{\ensuremath{\llbracket #1 \rrbracket}} 
\newcommand{\ps}[2][{}]{\ensuremath{\mathcal{P}_{#1}(#2)}} 
\newcommand{\mps}[2][{}]{\ensuremath{\mathcal{M}_{#1}(#2)}} 
\newcommand{\col}[2][{}]{\ensuremath{\mathrm{col}_{#1}(#2)}}

\newcommand{\Sym}{\ensuremath{\mathrm{Sym}}}

\title{State sums for some super quantum link invariants}
\allowdisplaybreaks
 \author{Louis-Hadrien Robert}
 \address{Université de Genève, rue du lièvre 2--4, 1227 Genève, Switzerland}
 \email{\href{mailto:louis-hadrien.robert@unige.ch}{louis-hadrien.robert@unige.ch}}
 \author{Emmanuel Wagner}
 \address{Univ Paris Diderot, IMJ-PRG, UMR 7586 CNRS, F-75013, Paris, France} 
\email{\href{mailto:emmanuel.wagner@imj-prg.fr}{emmanuel.wagner@imj-prg.fr}}
\tikzset{->-/.style={decoration={
  markings,
  mark=at position .5 with {\arrow{>}}},postaction={decorate}}}
\tikzset{-<-/.style={decoration={
  markings,
  mark=at position .5 with {\arrow{<}}},postaction={decorate}}}

\makeatletter
\DeclareFontFamily{OMX}{MnSymbolE}{}
\DeclareSymbolFont{MnLargeSymbols}{OMX}{MnSymbolE}{m}{n}
\SetSymbolFont{MnLargeSymbols}{bold}{OMX}{MnSymbolE}{b}{n}
\DeclareFontShape{OMX}{MnSymbolE}{m}{n}{
    <-6>  MnSymbolE5
   <6-7>  MnSymbolE6
   <7-8>  MnSymbolE7
   <8-9>  MnSymbolE8
   <9-10> MnSymbolE9
  <10-12> MnSymbolE10
  <12->   MnSymbolE12
}{}
\DeclareFontShape{OMX}{MnSymbolE}{b}{n}{
    <-6>  MnSymbolE-Bold5
   <6-7>  MnSymbolE-Bold6
   <7-8>  MnSymbolE-Bold7
   <8-9>  MnSymbolE-Bold8
   <9-10> MnSymbolE-Bold9
  <10-12> MnSymbolE-Bold10
  <12->   MnSymbolE-Bold12
}{}

\let\llangle\@undefined
\let\rrangle\@undefined
\DeclareMathDelimiter{\llangle}{\mathopen}
                     {MnLargeSymbols}{'164}{MnLargeSymbols}{'164}
\DeclareMathDelimiter{\rrangle}{\mathclose}
                     {MnLargeSymbols}{'171}{MnLargeSymbols}{'171}
\makeatother
 \newcommand{\digona}{\ensuremath{\vcenter{\hbox{\tikz[scale=0.4]{
\coordinate (B) at (0,0);
\coordinate (V1) at (0,0.5);
\coordinate (V2) at (0,2.5);
\coordinate (T) at (0,3);
\draw[white] (0, -0.5) -- (0, 3.5);
\draw[->] (B) -- (V1) node[midway, left] {\tiny{$m+n$}};
\draw[->] (V2) -- (T) node[midway, left] {\tiny{$m+n$}};
\draw[->] (V1) .. controls +(+0.5, +0.5) and +(+0.5, -0.5).. (V2) node[midway, right] {\tiny{$n$}};
\draw[->] (V1) .. controls +(-0.5, +0.5) and +(-0.5, -0.5).. (V2) node[midway, left] {\tiny{$m$}};
}}}}}

\newcommand{\digonaone}{\ensuremath{\vcenter{\hbox{\tikz[scale=0.4]{
\coordinate (B) at (0,0);
\coordinate (V1) at (0,0.5);
\coordinate (V2) at (0,2.5);
\coordinate (T) at (0,3);
\draw[white] (0, -0.5) -- (0, 3.5);
\draw[->] (B) -- (V1) node[midway, left] {\tiny{$m+1$}};
\draw[->] (V2) -- (T) node[midway, left] {\tiny{$m+1$}};
\draw[->] (V1) .. controls +(+0.5, +0.5) and +(+0.5, -0.5).. (V2) node[midway, right] {\tiny{$1$}};
\draw[->] (V1) .. controls +(-0.5, +0.5) and +(-0.5, -0.5).. (V2) node[midway, left] {\tiny{$m$}};
}}}}}

\newcommand{\digonbone}{\ensuremath{\vcenter{\hbox{\tikz[scale=0.4]{
\coordinate (B) at (0,0);
\coordinate (V1) at (0,0.5);
\coordinate (V2) at (0,2.5);
\coordinate (T) at (0,3);
\draw[white] (0, -0.5) -- (0, 3.5);
\draw[->] (B) -- (V1) node[midway, left] {\tiny{$m+1$}};
\draw[->] (V2) -- (T) node[midway, left] {\tiny{$m+1$}};
\draw[->] (V1) .. controls +(+0.5, +0.5) and +(+0.5, -0.5).. (V2) node[midway, right] {\tiny{$m$}};
\draw[->] (V1) .. controls +(-0.5, +0.5) and +(-0.5, -0.5).. (V2) node[midway, left] {\tiny{$1$}};
}}}}}

\newcommand{\verta}{\ensuremath{\vcenter{\hbox{\tikz[scale=0.4]{
\coordinate (B) at (0,0);
\coordinate (T) at (0,3);
\draw[white] (0, -0.5) -- (0, 3.5);
\draw[->] (B) -- (T) node[midway, right] {\tiny{$m+n$}};
}}}}}

\newcommand{\vertaone}{\ensuremath{\vcenter{\hbox{\tikz[scale=0.4]{
\coordinate (B) at (0,0);
\coordinate (T) at (0,3);
\draw[white] (0, -0.5) -- (0, 3.5);
\draw[->] (B) -- (T) node[midway, right] {\tiny{$m+1$}};
}}}}}

\newcommand{\digonb}{\ensuremath{\vcenter{\hbox{\tikz[scale=0.4]{
\coordinate (B) at (0,0);
\coordinate (V1) at (0,0.5);
\coordinate (V2) at (0,2.5);
\coordinate (T) at (0,3);
\draw[white] (0, -0.5) -- (0, 3.5);
\draw[->] (B) -- (V1) node[midway, left] {\tiny{$m$}};
\draw[->] (V2) -- (T) node[midway, left] {\tiny{$m$}};
\draw[<-] (V1) .. controls +(+0.5, +0.5) and +(+0.5, -0.5).. (V2) node[midway, right] {\tiny{$n$}};
\draw[->] (V1) .. controls +(-0.5, +0.5) and +(-0.5, -0.5).. (V2) node[midway, left] {\tiny{$m+n$}};
}}}}}

\newcommand{\digonbbbb}{\ensuremath{\vcenter{\hbox{\tikz[scale=0.4]{
\coordinate (B) at (0,0);
\coordinate (V1) at (0,0.5);
\coordinate (V2) at (0,2.5);
\coordinate (T) at (0,3);
\draw[white] (0, -0.5) -- (0, 3.5);
\draw[->] (B) -- (V1) node[midway, left] {\tiny{$m$}};
\draw[->] (V2) -- (T) node[midway, left] {\tiny{$m$}};
\draw[<-] (V1) .. controls +(-0.5, +0.5) and +(-0.5, -0.5).. (V2) node[midway, left] {\tiny{$n$}};
\draw[->] (V1) .. controls +(+0.5, +0.5) and +(+0.5, -0.5).. (V2) node[midway, right] {\tiny{$m+n$}};
}}}}}

\newcommand{\digonbb}{\ensuremath{\vcenter{\hbox{\tikz[scale=0.4]{
\coordinate (B) at (0,0);
\coordinate (V1) at (0,0.5);
\coordinate (V2) at (0,2.5);
\coordinate (T) at (0,3);
\draw[white] (0, -0.5) -- (0, 3.5);
\draw[->] (B) -- (V1) node[midway, left] {\tiny{$m$}};
\draw[->] (V2) -- (T) node[midway, left] {\tiny{$m$}};
\draw[<-] (V1) .. controls +(+0.5, +0.5) and +(+0.5, -0.5).. (V2) node[midway, right] {\tiny{$1$}};
\draw[->] (V1) .. controls +(-0.5, +0.5) and +(-0.5, -0.5).. (V2) node[midway, left] {\tiny{$m+1$}};
}}}}}

\newcommand{\digonbbb}{\ensuremath{\vcenter{\hbox{\tikz[scale=0.4]{
\coordinate (B) at (0,0);
\coordinate (V1) at (0,0.5);
\coordinate (V2) at (0,2.5);
\coordinate (T) at (0,3);
\draw[white] (0, -0.5) -- (0, 3.5);
\draw[->] (B) -- (V1) node[midway, left] {\tiny{$m$}};
\draw[->] (V2) -- (T) node[midway, left] {\tiny{$m$}};
\draw[<-] (V1) .. controls +(-0.5, +0.5) and +(-0.5, -0.5).. (V2) node[midway, left] {\tiny{$1$}};
\draw[->] (V1) .. controls +(+0.5, +0.5) and +(+0.5, -0.5).. (V2) node[midway, right] {\tiny{$m+1$}};
}}}}}

\newcommand{\vertb}{\ensuremath{\vcenter{\hbox{\tikz[scale=0.4]{
\coordinate (B) at (0,0);
\coordinate (T) at (0,3);
\draw[white] (0, -0.5) -- (0, 3.5);
\draw[->] (B) -- (T) node[midway, right] {\tiny{$m$}};
}}}}}

\newcommand{\stgamma}{\ensuremath{\vcenter{\hbox{\tikz[scale=0.3]{
\coordinate (B) at (0,0);
\coordinate (V1) at (0,1);
\coordinate (V2) at (1,2);
\coordinate (T1) at (-2,3);
\coordinate (T2) at (0,3);
\coordinate (T3) at (2,3);
\draw[>-] (B) -- (V1) node [at start, below] {\tiny{$i+j+k$}};
\draw[->] (V1) -- (T1) node [at end, above] {\tiny{$i$}};
\draw[->] (V1)  -- (V2) node[midway, right] {\tiny{$j+k$}};
\draw[->] (V2) -- (T2) node[at end, above] {\tiny{$j$}};
\draw[->] (V2) -- (T3) node[at end, above] {\tiny{$k$}};
}}}}}

\newcommand{\stgammaprime}{\ensuremath{\vcenter{\hbox{\tikz[scale=0.3]{
\coordinate (B) at (0,0);
\coordinate (V1) at (0,1);
\coordinate (V2) at (-1,2);
\coordinate (T1) at (-2,3);
\coordinate (T2) at (0,3);
\coordinate (T3) at (2,3);
\draw[>-] (B) -- (V1) node [at start, below] {\tiny{$i+j+k$}};
\draw[->] (V1) -- (T3) node [at end, above] {\tiny{$k$}};
\draw[->] (V1)  -- (V2) node[midway, left] {\tiny{$i+j$}};
\draw[->] (V2) -- (T1) node[at end, above] {\tiny{$i$}};
\draw[->] (V2) -- (T2) node[at end, above] {\tiny{$j$}};
}}}}}

\newcommand{\stgammar}{\ensuremath{\vcenter{\hbox{\tikz[scale=0.3, yscale = -1]{
\coordinate (B) at (0,0);
\coordinate (V1) at (0,1);
\coordinate (V2) at (1,2);
\coordinate (T1) at (-2,3);
\coordinate (T2) at (0,3);
\coordinate (T3) at (2,3);
\draw[<-] (B) -- (V1) node [at start, above] {\tiny{$i+j+k$}};
\draw[-<] (V1) -- (T1) node [at end, below] {\tiny{$i$}};
\draw[-<] (V1)  -- (V2) node[midway, right] {\tiny{$j+k$}};
\draw[-<] (V2) -- (T2) node[at end, below] {\tiny{$j$}};
\draw[-<] (V2) -- (T3) node[at end, below] {\tiny{$k$}};
}}}}}

\newcommand{\stgammaprimer}{\ensuremath{\vcenter{\hbox{\tikz[scale=0.3, yscale = -1]{
\coordinate (B) at (0,0);
\coordinate (V1) at (0,1);
\coordinate (V2) at (-1,2);
\coordinate (T1) at (-2,3);
\coordinate (T2) at (0,3);
\coordinate (T3) at (2,3);
\draw[<-] (B) -- (V1) node [at start, above] {\tiny{$i+j+k$}};
\draw[-<] (V1) -- (T3) node [at end, below] {\tiny{$k$}};
\draw[-<] (V1)  -- (V2) node[midway, left] {\tiny{$i+j$}};
\draw[-<] (V2) -- (T1) node[at end, below] {\tiny{$i$}};
\draw[-<] (V2) -- (T2) node[at end, below] {\tiny{$j$}};
}}}}} \newcommand{\squarea}{\ensuremath{\vcenter{\hbox{\tikz[scale=0.4]{
\coordinate (B1) at (-1,0);
\coordinate (B2) at (1,0);
\coordinate (C1) at (-1,1);
\coordinate (D1) at (-1,2);
\coordinate (C2) at (1,1);
\coordinate (D2) at (1,2);
\coordinate (T1) at (-1,3);
\coordinate (T2) at (1,3);
\draw[->] (B1) -- (C1) node[at start, below] {\tiny{$1$}};
\draw[->] (D1) -- (C1) node[midway, left   ] {\tiny{$m$}};
\draw[->-] (D1) -- (T1) node[at end , above ] {\tiny{$1$}};
\draw[->] (C2) -- (B2) node[at end, below] {\tiny{$m$}};
\draw[->-] (C2) -- (D2) node[midway, right] {\tiny{$1$}};
\draw[->] (T2) -- (D2) node[at start, above] {\tiny{$m$}};
\draw[->-] (D2) -- (D1) node[midway, above] {\tiny{$m+1$}};
\draw[->-] (C1) -- (C2) node[midway, below] {\tiny{$m+1$}};
}}}}}

\newcommand{\squaremm}{\ensuremath{\vcenter{\hbox{\tikz[scale=0.4]{
\coordinate (B1) at (-1,0);
\coordinate (B2) at (1,0);
\coordinate (C1) at (-1,1);
\coordinate (D1) at (-1,2);
\coordinate (C2) at (1,1);
\coordinate (D2) at (1,2);
\coordinate (T1) at (-1,3);
\coordinate (T2) at (1,3);
\draw[->] (B1) -- (C1) node[at start, below] {\tiny{$m$}};
\draw[-<-] (D1) -- (C1) node[midway, left   ] {\tiny{$m+1$}};
\draw[->] (D1) -- (T1) node[at end , above ] {\tiny{$m$}};
\draw[->] (C2) -- (B2) node[at end, below] {\tiny{$m$}};
\draw[-<-] (C2) -- (D2) node[midway, right] {\tiny{$m+1$}};
\draw[->] (T2) -- (D2) node[at start, above] {\tiny{$m$}};
\draw[-<-] (D2) -- (D1) node[midway, above] {\tiny{$1$}};
\draw[-<-] (C1) -- (C2) node[midway, below] {\tiny{$1$}};
}}}}}

\newcommand{\crossposone}{\ensuremath{\vcenter{\hbox{\tikz[font=\tiny, scale=0.4]{
\coordinate (B1) at (-1,-1);
\coordinate (B2) at ( 1,-1);
\coordinate (T1) at (-1, 1);
\coordinate (T2) at ( 1, 1);
\draw [->] (B2) -- (T1) node[pos=0.8, left] {$1$};
\fill[white] (0,0) circle (2mm);
\draw [->] (B1) -- (T2) node[pos=0.8, right] {$1$};
}}}}}

\newcommand{\crossnegone}{\ensuremath{\vcenter{\hbox{\tikz[font=\tiny, scale=0.4]{
\coordinate (B1) at (-1,-1);
\coordinate (B2) at ( 1,-1);
\coordinate (T1) at (-1, 1);
\coordinate (T2) at ( 1, 1);
\draw [->] (B1) -- (T2) node[pos=0.8, right] {$1$};
\fill[white] (0,0) circle (2mm);
\draw [->] (B2) -- (T1) node[pos=0.8, left] {$1$};
}}}}}

\newcommand{\twovertsmoothone}{\ensuremath{\vcenter{\hbox{\tikz[font=\tiny, scale=0.4]{
\coordinate (B1) at (-1,-1);
\coordinate (B2) at ( 1,-1);
\coordinate (T1) at (-1, 1);
\coordinate (T2) at ( 1, 1);
\draw [->] (B1) .. controls +( 0.3, 0.3) and  +( 0.3, -0.3) .. (T1) node[pos=0.5, left] {$1$};
\draw [->] (B2) .. controls +(-0.3, 0.3) and  +(-0.3, -0.3) .. (T2) node[pos=0.5, right] {$1$};
}}}}}

\newcommand{\dumbleone}{\ensuremath{\vcenter{\hbox{\tikz[font=\tiny, scale=0.4]{
\coordinate (B1) at (-1,-1);
\coordinate (B2) at ( 1,-1);
\coordinate (MB) at ( 0,-0.5);
\coordinate (MT) at ( 0, 0.5);
\coordinate (T1) at (-1, 1);
\coordinate (T2) at ( 1, 1);
\draw [->-] (B1) -- (MB)  node [font=\tiny, below, midway] {$1$};
\draw [->-] (B2) -- (MB)  node [font=\tiny, below, midway] {$1$};
\draw [->-] (MB) -- (MT) node [font=\tiny, right, midway] {$2$};
\draw [-<-] (T1) -- (MT)  node [font=\tiny, above, midway] {$1$};
\draw [-<-] (T2) -- (MT)  node [font=\tiny, above, midway] {$1$};
}}}}}

\newcommand{\squaremmm}{\ensuremath{\vcenter{\hbox{\tikz[scale=0.4]{
\coordinate (B1) at (-1,0);
\coordinate (B2) at (1,0);
\coordinate (C1) at (-1,1);
\coordinate (D1) at (-1,2);
\coordinate (C2) at (1,1);
\coordinate (D2) at (1,2);
\coordinate (T1) at (-1,3);
\coordinate (T2) at (1,3);
\draw[->] (B1) -- (C1) node[at start, below] {\tiny{$m$}};
\draw[-<-] (D1) -- (C1) node[midway, left   ] {\tiny{$m-1$}};
\draw[->] (D1) -- (T1) node[at end , above ] {\tiny{$m$}};
\draw[->] (C2) -- (B2) node[at end, below] {\tiny{$m$}};
\draw[-<-] (C2) -- (D2) node[midway, right] {\tiny{$m-1$}};
\draw[->] (T2) -- (D2) node[at start, above] {\tiny{$m$}};
\draw[->-] (D2) -- (D1) node[midway, above] {\tiny{$1$}};
\draw[->-] (C1) -- (C2) node[midway, below] {\tiny{$1$}};
}}}}}

\newcommand{\twoverta}{\ensuremath{\vcenter{\hbox{\tikz[scale=0.4]{
\coordinate (B1) at (-1,0);
\coordinate (T1) at (-1,3);
\coordinate (B2) at (1,0);
\coordinate (T2) at (1,3);
\draw[->] (B1) -- (T1) node[midway, left] {\tiny{$1$}};
\draw[->] (T2) -- (B2) node[midway, right] {\tiny{$m$}};
}}}}}

\newcommand{\twovertmm}{\ensuremath{\vcenter{\hbox{\tikz[scale=0.4]{
\coordinate (B1) at (-1,0);
\coordinate (T1) at (-1,3);
\coordinate (B2) at (1,0);
\coordinate (T2) at (1,3);
\draw[->] (B1) -- (T1) node[midway, left] {\tiny{$m$}};
\draw[->] (T2) -- (B2) node[midway, right] {\tiny{$m$}};
}}}}}

\newcommand{\twovertddd}{\ensuremath{\vcenter{\hbox{\tikz[scale=0.4]{
\coordinate (B1) at (-1,0);
\coordinate (T1) at (-1,3);
\coordinate (B2) at (1,0);
\coordinate (T2) at (1,3);
\draw[->] (B1) -- (T1) node[midway, left] {\tiny{$n$}};
\draw[<-] (T2) -- (B2) node[midway, right] {\tiny{$n+k$}};
}}}}}

\newcommand{\doubleYa}{\ensuremath{\vcenter{\hbox{\tikz[scale=0.4]{
\coordinate (B1) at (-1,0);
\coordinate (T1) at (-1,3);
\coordinate (C) at (0,1);
\coordinate (D) at (0,2);
\coordinate (B2) at (1,0);
\coordinate (T2) at (1,3);
\draw[->] (B1) -- (C) node[at start, below] {\tiny{$1$}};
\draw[->] (C) -- (B2) node[at end, below] {\tiny{$m$}};
\draw[->] (D) -- (C) node[midway, left] {\tiny{$m-1$}};
\draw[->] (T2) -- (D) node[at start, above] {\tiny{$m$}};
\draw[->] (D) -- (T1) node[at end, above] {\tiny{$1$}};
}}}}}

\newcommand{\squareb}{\ensuremath{\vcenter{\hbox{\tikz[scale=0.55]{
\coordinate (B1) at (-1,0);
\coordinate (B2) at (1,0);
\coordinate (C1) at (-1,1);
\coordinate (D1) at (-1,2);
\coordinate (C2) at (1,1);
\coordinate (D2) at (1,2);
\coordinate (T1) at (-1,3);
\coordinate (T2) at (1,3);
\draw[->] (B1) -- (C1) node[at start, below] {\tiny{$1$}};
\draw[->] (C1) -- (D1) node[midway, left   ] {\tiny{$l+n$}};
\draw[->] (D1) -- (T1) node[at end , above ] {\tiny{$l$}};
\draw[->] (B2) -- (C2) node[at start, below] {\tiny{$m+l-1$}};
\draw[->] (C2) -- (D2) node[midway, right] {\tiny{$m-n$}};
\draw[->] (D2) -- (T2) node[at end, above] {\tiny{$m$}};
\draw[->] (D1) -- (D2) node[midway, above] {\tiny{$n$}};
\draw[->] (C2) -- (C1) node[midway, below] {\tiny{$l+n-1$}};
}}}}}

\newcommand{\squarec}{\ensuremath{\vcenter{\hbox{\tikz[xscale=0.65, yscale=0.55]{
\coordinate (B1) at (-1,0);
\coordinate (B2) at (1,0);
\coordinate (C1) at (-1,1.1);
\coordinate (D1) at (-1,1.9);
\coordinate (C2) at (1,0.9);
\coordinate (D2) at (1,2.1);
\coordinate (T1) at (-1,3);
\coordinate (T2) at (1,3);
\draw[->] (B1) -- (C1) node[at start, below] {\tiny{$n$}};
\draw[->] (C1) -- (D1) node[midway, left   ] {\tiny{$n+k $}};
\draw[->] (D1) -- (T1) node[at end , above ] {\tiny{$m$}};
\draw[->] (B2) -- (C2) node[at start, below] {\tiny{$m+l$}};
\draw[->] (C2) -- (D2) node[midway, right] {\tiny{$m+l-k$}};
\draw[->] (D2) -- (T2) node[at end, above] {\tiny{$n+l$}};
\draw[->] (D1) -- (D2) node[midway, above] {\tiny{$n+k-m$}};
\draw[->] (C2) -- (C1) node[midway, below] {\tiny{$k$}};
}}}}}

\newcommand{\squareccc}{\ensuremath{\vcenter{\hbox{\tikz[xscale=0.65, yscale=0.55]{
\coordinate (B1) at (-1,0);
\coordinate (B2) at (1,0);
\coordinate (C1) at (-1,1.1);
\coordinate (D1) at (-1,1.9);
\coordinate (C2) at (1,0.9);
\coordinate (D2) at (1,2.1);
\coordinate (T1) at (-1,3);
\coordinate (T2) at (1,3);
\draw[->] (B1) -- (C1) node[at start, below] {\tiny{$n$}};
\draw[->] (C1) -- (D1) node[midway, left   ] {\tiny{$n+1 $}};
\draw[->] (D1) -- (T1) node[at end , above ] {\tiny{$n$}};
\draw[->] (B2) -- (C2) node[at start, below] {\tiny{$n+k$}};
\draw[->] (C2) -- (D2) node[midway, right] {\tiny{$n+k-1$}};
\draw[->] (D2) -- (T2) node[at end, above] {\tiny{$n+k$}};
\draw[->] (D1) -- (D2) node[midway, above] {\tiny{$1$}};
\draw[->] (C2) -- (C1) node[midway, below] {\tiny{$1$}};
}}}}}

\newcommand{\squared}{\ensuremath{\vcenter{\hbox{\tikz[yscale=0.55, xscale=0.65]{
\coordinate (B1) at (-1,0);
\coordinate (B2) at (1,0);
\coordinate (C1) at (-1,0.9);
\coordinate (D1) at (-1,2.1);
\coordinate (C2) at (1,1.1);
\coordinate (D2) at (1,1.9);
\coordinate (T1) at (-1,3);
\coordinate (T2) at (1,3);
\draw[->] (B1) -- (C1) node[at start, below] {\tiny{$n$}};
\draw[->] (C1) -- (D1) node[midway, left   ] {\tiny{$m-j$}};
\draw[->] (D1) -- (T1) node[at end , above ] {\tiny{$m$}};
\draw[->] (B2) -- (C2) node[at start, below] {\tiny{$m+l$}};
\draw[->] (C2) -- (D2) node[midway, right] {\tiny{$n+l+j$}};
\draw[->] (D2) -- (T2) node[at end, above] {\tiny{$n+l$}};
\draw[->] (D2) -- (D1) node[midway, above] {\tiny{$j$}};
\draw[->] (C1) -- (C2) node[midway, below] {\tiny{$n+j-m$}};
}}}}}

\newcommand{\squareddd}{\ensuremath{\vcenter{\hbox{\tikz[yscale=0.55, xscale=0.65]{
\coordinate (B1) at (-1,0);
\coordinate (B2) at (1,0);
\coordinate (C1) at (-1,0.9);
\coordinate (D1) at (-1,2.1);
\coordinate (C2) at (1,1.1);
\coordinate (D2) at (1,1.9);
\coordinate (T1) at (-1,3);
\coordinate (T2) at (1,3);
\draw[->] (B1) -- (C1) node[at start, below] {\tiny{$n$}};
\draw[->] (C1) -- (D1) node[midway, left   ] {\tiny{$n-1$}};
\draw[->] (D1) -- (T1) node[at end , above ] {\tiny{$n$}};
\draw[->] (B2) -- (C2) node[at start, below] {\tiny{$n+k$}};
\draw[->] (C2) -- (D2) node[midway, right] {\tiny{$n+k+1$}};
\draw[->] (D2) -- (T2) node[at end, above] {\tiny{$n+k$}};
\draw[->] (D2) -- (D1) node[midway, above] {\tiny{$1$}};
\draw[->] (C1) -- (C2) node[midway, below] {\tiny{$1$}};
}}}}}

\newcommand{\squaree}{\ensuremath{\vcenter{\hbox{\tikz[yscale=0.55, xscale=0.65]{
\coordinate (B1) at (-1,0);
\coordinate (B2) at (1,0);
\coordinate (C1) at (-1,1);
\coordinate (D1) at (-1,2);
\coordinate (C2) at (1,1);
\coordinate (D2) at (1,2);
\coordinate (T1) at (-1,3);
\coordinate (T2) at (1,3);
\draw[->] (B1) -- (C1) node[at start, below] {\tiny{$k+s$}};
\draw[->] (C1) -- (D1) node[midway, left   ] {\tiny{$k$}};
\draw[->] (D1) -- (T1) node[at end , above ] {\tiny{$k-r$}};
\draw[->] (B2) -- (C2) node[at start, below] {\tiny{$l-s$}};
\draw[->] (C2) -- (D2) node[midway, right] {\tiny{$l$}};
\draw[->] (D2) -- (T2) node[at end, above] {\tiny{$l+r$}};
\draw[<-] (D2) -- (D1) node[midway, above] {\tiny{$r$}};
\draw[->] (C1) -- (C2) node[midway, below] {\tiny{$s$}};
}}}}}

\newcommand{\webHe}{\ensuremath{\vcenter{\hbox{\tikz[yscale=0.55, xscale=0.65]{
\coordinate (B1) at (-1,0);
\coordinate (B2) at (1,0);
\coordinate (C1) at (-1,1.5);
\coordinate (D1) at (-1,1.5);
\coordinate (C2) at (1,1.5);
\coordinate (D2) at (1,1.5);
\coordinate (T1) at (-1,3);
\coordinate (T2) at (1,3);
\draw[->] (B1) -- (C1) node[at start, below] {\tiny{$k+s$}};
\draw[->] (D1) -- (T1) node[at end , above ] {\tiny{$k-r$}};
\draw[->] (B2) -- (C2) node[at start, below] {\tiny{$l-s$}};
\draw[->] (D2) -- (T2) node[at end, above] {\tiny{$l+r$}};
\draw[->] (C1) -- (C2) node[midway, below] {\tiny{$r+s$}};
}}}}}

\newcommand{\doubleYb}{\ensuremath{\vcenter{\hbox{\tikz[scale=0.4]{
\coordinate (B1) at (-1,0);
\coordinate (T1) at (-1,3);
\coordinate (C) at (0,1);
\coordinate (D) at (0,2);
\coordinate (B2) at (1,0);
\coordinate (T2) at (1,3);
\draw[->] (B1) -- (C) node[at start, below] {\tiny{$1$}};
\draw[<-] (C) -- (B2) node[at end, below] {\tiny{$m+l-1$}};
\draw[<-] (D) -- (C) node[midway, left] {\tiny{$l+m$}};
\draw[<-] (T2) -- (D) node[at start, above] {\tiny{$m$}};
\draw[->] (D) -- (T1) node[at end, above] {\tiny{$l$}};
}}}}}

\newcommand{\bigHb}{\ensuremath{\vcenter{\hbox{\tikz[scale=0.4]{
\coordinate (B1) at (-1,0);
\coordinate (T1) at (-1,3);
\coordinate (M1) at (-1,1.5);
\coordinate (M2) at (1,1.5);
\coordinate (B2) at (1,0);
\coordinate (T2) at (1,3);
\draw[->] (B1) -- (M1) node[at start, below] {\tiny{$1$}};
\draw[->] (B2) -- (M2) node[at start, below] {\tiny{$m+l-1$}};
\draw[->] (M2) -- (M1) node[midway, above] {\tiny{$l-1$}};
\draw[->] (M2) -- (T2) node[at end, above] {\tiny{$m$}};
\draw[->] (M1) -- (T1) node[at end, above] {\tiny{$l$}};
}}}}}

\begin{document}
\maketitle

\begin{abstract}
We present state sums for quantum link invariants arising from the representation theory of $U_q(\gll_{N|M})$. We investigate the case of the $N$th exterior power of the standard representation of $U_q(\gll_{N|1})$ and explicit the relation with Kashaev invariants.
\end{abstract}

\tableofcontents

\section*{Introduction}
\label{sec:introduction}

In its 1990 ICM paper \cite{MR1159255}, Turaev emphazided the proeminent role of state sum models in low dimensional topology. Models of this kind were and remain an important part of what is called quantum topology.  They are closely related to mathematical physics, quantum algebra and statistical mechanics. The most famous example of state sum model in low dimensional topology is probably the Kauffman bracket \cite{MR899057}. Not only demonstrates it how methods of statistical mechanics and ideas of quantum field theory can be relevant in this subject as  Witten explained \cite{MR990772}, but it is also the first step in a recent major development of quantum topology: categorification of quantum invariants. Indeed, Khovanov homology \cite{MR1740682} which categorifies the Jones polynomial is clearly inspired by the combinatorial definition of the Jones polynomial given by the Kauffman bracket.\\

In this paper, we give state sum models computing quantum invariants arising from the representation theory of the super quantum group $U_q(\gll_{N|M})$. We believe these state sums will be useful in the quest of categorifying these quantum invariants. Some of the invariants described by these state sums have a non-semi-simple nature. The prototypical example of such non-semi-simple invariants is the Alexander--Conway polynomial.\\

Let us discuss---with a slightly biased point of view---some of the state sums and categorifications of the Alexander--Conway polynomial. 
The first categorification of the Alexander provided by the knot Floer homology of Ozsv\'ath--Szab\'o and Rasmussen \cite{OS, Rasmussen} uses its interpretation as Reidemeister torsion. In a combinatorial version of Manolescu--Ozsv\'ath--Szab\'o--Thurston using grid diagrams, they used explicitely a determinantal description \cite[Appendix]{MR2372850}. They were many attempts to have a direct combinatorial approach using the state model developped by Kauffman in \cite{MR712133}, but finally Ozsv\'ath--Szab\'o made the connection in \cite{MR2574747} using a twisted version of the knot Floer homology.
In \cite{RW3}, the authors of this paper provided a combinatorial categorification of the Alexander--Conway polynomial of knots starting from a representation theoric interpretation (see \cite{EPV} for the connection between knot Floer homology and representation theory). The state sum model underlying this last categorification is one of the example of the present paper (see also \cite{Viro, MR3319619}).\\

The paper \cite{MR1659228} by Murakami--Ohtsuki--Yamada is  another prototypical example of a state sum model which found its full relevance and importance in the realm of the categorification process. They give an elementary description of Re\-she\-ti\-khin--Turaev invariant associated with exterior powers of the standard representation of $U_q(\gll_{N})$. The description of these quantum invariants by Murakami--Ohtsuki--Yamada proceeds in two steps. First, a link diagram is expressed as a formal linear combination of planar trivalent graphs (see also Kuperberg \cite{MR1403861}), then a positive state sum is provided for these graphs. Both of these steps were key ingredients for the categorification of $U_q(\gll_N)$-quantum invariants by Khovanov--Rozansky \cite{MR2391017} (see also  Mackaay--Sto{\v{s}}i\'c--Vaz \cite{MR2491657} and  Wu \cite{pre06302580}, Yonezawa \cite{yonezawa2011}, Mazorchuk--Stroppel\cite{zbMATH05656519}  and Sussan \cite{MR2710319} for colored versions).\\

State sum models should continue to be of importance and influencal in quantum topology. For instance, they play a direct role in another version of the categorification of the previous invariants by the authors in \cite{RW1}. A state sum model for closed 2-dimensional CW complexes called foams is given to produce a trivalent TQFT. This state sum model was reinvested by Khovanov and the first author \cite{KR1} to investigate the four color theorem.\\

The present paper is firstly concerned with state sum models for exterior powers of the standard representation of $U_q(\gll_{N|M})$. In particular the case $N=0$, corresponds to invariants associated with symmetric powers of the standard representation of $U_q(\gll_{M})$. These invariants have been categorified in \cite{RW2, queffelec2018annular, MR3709661}. As reproved in the present paper, the invariants only depend on $N-M$, providing in fact different state sum models for the same invariants. It is conjectured in \cite{GGS} that these various state sums could be categorified yielding potentially non-equivalent homological invariants.\\

Secondly, the paper investigates further the invariants associated with the $N$th exterior power of the standard representation of $U_q(\gll_{N|1})$ which have vanishing quantum dimensions. These invariants can be renormalized, see \cite{MR2640994, 2015arXiv150603329Q}. This paper provides a direct proof that they can be renormalized giving explicitly a state sum model. It also explains how these invariants are related to  Kashaev invariants \cite{MR1341338} (see  \cite{MR2468374} for a different proof of this latter fact). As ADO invariants \cite{MR1164114}, this family of non-semisimple invariants provides  interesting family of invariants generalizing the Alexander polynomial  containing  Kashaev invariants. The Alexander polynomial has been categorified in this framework \cite{RW3} and the present state sum models provide a first step in the categorification of this family of invariants.\\

\subsection*{Outline of the paper}
\begin{itemize}
\item Section~\ref{sec:qi} is devoted to some combinatorics of $q$-binomials used in the rest of the paper.
\item Section~\ref{sec:colorings-moy-graphs} gives the state sum models for planar graphs. 
\item Section~\ref{sec:link-invariants} extend the state sum models to link diagrams. 
\item Section~\ref{sec:non-semi-simple} investigates the non-semisimple case of the $N$-th exterior power of the standard representation of $U_q(\gll_{N|1})$ and makes the connection with the Kashaev invariants.
\item Appendix~\ref{sec:moy-graphs-an} provides the representation theoric background; in particular the explicit maps providing the corresponding Reshetikhin--Turaev functors.
\end{itemize}

\section{Some $q$-identities}
\label{sec:qi}
In this section, we give some useful identities on quantum integers and quantum binomials. See \cite{MR1865777} for a more detailed account. The first two lemmas and three next corollaries are obtained by direct computations or easy inductions and we omit their proofs. Unless otherwise specified, $q$ is a formal parameter.

\begin{dfn}
  \label{dfn:quantum-numbers}
  Let $n$ be an integer, define the \emph{quantum integer $[n]$} by the following formula:
\[[n]:=\frac{q^n -q^{-n}}{q-q^{-1}}=
\begin{cases}
  \sum_{i=1}^{n} q^{-1-n +2i} & \textrm{if $n> 0$,} \\
  0 & \textrm{if $n=0$ and} \\
  \sum_{i=-n}^{-1} -q^{1+n +2i} & \textrm{if $n<0$.}
\end{cases}
\]
If $k$ and $n$ are two integers, define the \emph{quantum binomial $
\begin{bmatrix}
  n \\ k
\end{bmatrix}
$} by the following formula:
\[
\begin{bmatrix}
  n \\ k
\end{bmatrix} =
\begin{cases}
  \frac{\prod_{i=0}^{k-1}[n-i]}{\prod_{i=1}^k[i]} & \textrm{if $k\geq 0$,} \\
  0 &\textrm{else.}
\end{cases}
\]
\end{dfn}

\begin{rmk}
  For all integers $k$ and $n$, one has:
  \[
[-n] = -[n] \qquad \textrm{and} \qquad
\begin{bmatrix}
  -n \\ k
\end{bmatrix}
 = (-1)^k
 \begin{bmatrix}
   n+k -1 \\ k
 \end{bmatrix}.
  \]
\end{rmk}

\begin{lem}
  \label{lem:add-integers}
  Let $n$ and $m$ be two integers, the following identity holds:
\[
[m+n] = q^{-n}[m] + q^{m}[n] = q^{n}[m] + q^{-m}[n].
\]
\end{lem}

\begin{cor}
  \label{cor:sum-integers}
  Let $k$ be an integer and $(a_h)_{1\leq h \leq k}$ be a collection of $k$ integers, then the following identity holds:
\[
\left[\sum_{h=1}^k a_h\right]=
\sum_{h=1}^k q^{\sum_{i=1}^{h-1} a_i - \sum_{i=h+1}^k a_i} [a_h] =
\sum_{h=1}^k q^{-\sum_{i=1}^{h-1} a_i + \sum_{i=h+1}^k a_i} [a_h].
\]
\end{cor}

\begin{cor}
  \label{lem:diff-prod}
  Let $k$, $m$ and $n$ be three integers, then the following identity holds:
\[
[m+k][n] -[m][n+k] = [k][n -m].
\]
\end{cor}

\begin{lem}
  \label{lem:pascal}
  Let $n$ and $k$ be two integers, then the following identity holds:
  \begin{align*}
  \begin{bmatrix}
    n \\k
  \end{bmatrix} =
  q^k
  \begin{bmatrix}
    n - 1  \\ k
  \end{bmatrix} +
  q^{k-n}
  \begin{bmatrix}
    n - 1  \\ k -1
  \end{bmatrix}.
\end{align*}
\end{lem}

\begin{cor}\label{cor:anti-pascal}
  Let $n$ and $k$ be two integers with $k\geq 0$, then the following identity holds:
  \begin{align*}
  \begin{bmatrix}
    n \\k
  \end{bmatrix} =
  \sum_{i=0}^{k}(-1)^{k-i}q^{(i-k)(n-1) - k}
  \begin{bmatrix}
    n +1 \\i
  \end{bmatrix}.
\end{align*}
\end{cor}

A version of the next proposition appears in \cite[page 23]{MR1865777} with a different proof only valid for $n\geq 0$ and $m\geq 0$. Hence, we provide a complete proof.

\begin{prop}\label{prop:dec-binomial}
  Let $n_1$, $n_2$ and $k$ be three integers. The following identity holds:
  \begin{align*}
    \begin{bmatrix}
      n_1 + n_2 \\k
    \end{bmatrix}
    = \sum_{\substack{k_1 + k_2 = k \\ k_1, k_2 \geq 0}}
q^{n_1k_2 - n_2k_1}
    \begin{bmatrix}
      n_1 \\ k_1
    \end{bmatrix}
    \begin{bmatrix}
      n_2 \\ k_2
    \end{bmatrix}.
  \end{align*}
\end{prop}
\begin{proof}
  This statement is proved by induction on $n_1$. If $n_1 =0$  this is obvious. One needs to consider the case $n_1\geq 0$ and the case $n_1 \leq 0$.

  Suppose that $n_1\geq 0$ and that the statement holds for $n_1$. We use Lemma~\ref{lem:pascal} to compute:
  \begin{align*}
    &\sum_{\substack{k_1 + k_2 = k \\ k_1, k_2 \geq 0}}
q^{(n_1+1)k_2 - n_2k_1}
    \begin{bmatrix}
      n_1 +1 \\ k_1
    \end{bmatrix}
    \begin{bmatrix}
      n_2 \\ k_2
    \end{bmatrix}   \\
    &\qquad=
    \sum_{\substack{k_1 + k_2 = k \\ k_1, k_2 \geq 0}}q^{(n_1+1)k_2 - n_2k_1}
    \left(q^{k_1}
    \begin{bmatrix}
      n_1  \\ k_1
    \end{bmatrix}
    +q^{k_1 - (n_1 +1)}
    \begin{bmatrix}
      n_1  \\ k_1 -1
    \end{bmatrix}
    \right)
    \begin{bmatrix}
      n_2 \\ k_2
    \end{bmatrix}   \\
        & \qquad =
    q^{k}\sum_{\substack{k_1 + k_2 = k \\ k_1, k_2 \geq 0}}q^{n_1k_2- n_2k_1}
    \begin{bmatrix}
      n_1  \\ k_1
    \end{bmatrix}
    \begin{bmatrix}
      n_2 \\ k_2
    \end{bmatrix}
    \\ & \hspace{2cm}+ 
     q^{k-n_1-n_2-1}\sum_{\substack{k_1 + k_2 = k \\ k_1, k_2 \geq 0}}q^{n_1k_2 - n_2(k_1-1)}
    \begin{bmatrix}
      n_1  \\ k_1 -1
    \end{bmatrix}
    \begin{bmatrix}
      n_2 \\ k_2
    \end{bmatrix}   \\
    &\qquad =
    q^k
    \begin{bmatrix}
      n_1 + n_2 \\k
    \end{bmatrix}
    +
    q^{k-(n_1+n_2+1)}
    \begin{bmatrix}
      n_1 + n_2 \\ k-1
    \end{bmatrix} \\
    &\qquad =
    \begin{bmatrix}
      n_1 + n_2 +1 \\k
    \end{bmatrix}.
  \end{align*}
  Hence the statement holds for $n_1+1$, and therefore for all $n_1 \in \ZZ_\geq0$.

  Suppose that $n_1 < 0$ and that the statement holds for $n_1+1$. We use Corollary~\ref{cor:anti-pascal} to compute: 
  \begin{align*}
    &\sum_{\substack{k_1 + k_2 = k \\ k_1, k_2 \geq 0}}
    q^{n_1k_2 - n_2k_1}
    \begin{bmatrix}
      n_1  \\ k_1
    \end{bmatrix}
    \begin{bmatrix}
      n_2 \\ k_2
    \end{bmatrix}   \\
    &\qquad=
    \sum_{\substack{k_1 + k_2 = k \\ k_1, k_2 \geq 0}}
    q^{n_1k_2 - n_2k_1}
  \sum_{i=1}^{k_1}(-1)^{k_1-i}q^{(i-k_1)(n_1-1) - k_1}
    \begin{bmatrix}
      n_1 +1 \\ i
    \end{bmatrix}
    \begin{bmatrix}
      n_2 \\ k_2
    \end{bmatrix}   \\
   &\qquad= \sum_{j=0}^k\sum_{\substack{j_1 + j_2 = j \\ j_1, j_2 \geq 0}}
    (-1)^{k-j_1 -j_2}q^{(n_1j_2- n_2k + n_2j_2) + (j-k)(n_1-1) - k +j_2}
    \begin{bmatrix}
      n_1 +1 \\ j_1
    \end{bmatrix}
    \begin{bmatrix}
      n_2 \\ j_2
    \end{bmatrix}   \\
   &\qquad= \sum_{j=0}^k (-1)^{k-j} q^{(j-k)(n-1) -k}\sum_{\substack{j_1 + j_2 = j \\ j_1, j_2 \geq 0}}
    q^{(n_1+1)j_2- n_2j_1}
    \begin{bmatrix}
      n_1 +1 \\ j_1
    \end{bmatrix}
    \begin{bmatrix}
      n_2 \\ j_2
    \end{bmatrix}   \\
    &\qquad= \sum_{j=0}^k (-1)^{k-j} q^{(j-k)(n_1 + n_2 -1) -k}
    \begin{bmatrix}
      n_1 + n_2 +1 \\j 
    \end{bmatrix}\\
    & \qquad =
    \begin{bmatrix}
      n_1 + n_2 \\k
    \end{bmatrix}.
  \end{align*}
  Hence, the statement holds for $n_1$ and therefore for all negative integers.
\end{proof}

\begin{dfn}
  Let $X$ be a set.
  A \emph{multi-subset  of $X$} is an application $Y: X\to \ZZ_{\geq 0}$. If $\sum_{x\in X}Y(x)< \infty$, the multi-subset $Y$ is said to be \emph{finite} and the sum is its \emph{cardinal} (denoted by $\#Y$). If $x$ is an element of $X$, the number $Y(x)$ is the \emph{multiplicity of $x$ in $Y$}. 
An element $x$ of $X$ is \emph{in} $Y$, if its multiplicity in $Y$ is greater than or equal to $1$ (we write $x \in Y$). Let $Y_1$ and $Y_2$ be two multi-subsets of a set $X$.
\begin{itemize}
\item The \emph{disjoint union} of $Y_1$ and $Y_2$ (denoted $Y_1\sqcup
  Y_2$) is the multi-subset $Y_1+Y_2$.
\item The \emph{union} of $Y_1$ and
  $Y_2$ (denoted $Y_1\cup Y_2$) is the multi-subset $\max (Y_1,Y_2)$.
\item The \emph{intersection} of $Y_1$ and $Y_2$ (denoted $Y_1\cap Y_2$)
  is the multi-subset $\min (Y_1,Y_2)$.
\end{itemize}
\end{dfn}

\begin{rmk}
  A subset $Y$ of $X$ can be thought of as a multi-subset of $X$ by identifying it with its characteristic function. 
\end{rmk}

\begin{notation}
  Let $X$ be a set. The power set of $X$ is denoted by $\ps{X}$. The set of finite multi-subsets of $X$ is denoted by $\mps{X}$. For any non-negative integer $k$, the set of subsets (\resp{}multi-subsets) of $X$ of cardinal $k$ is denoted by $\ps[k]{X}$ (\resp$\mps[k]{X}$).
\end{notation}

\begin{dfn}
  Let $X$ be an ordered set and $Y$ a multi-subset of $X$. The \emph{degree of $Y$ in $X$} is the integer $\deg[X]{Y}$ given by the following formula:
  \begin{align*}
    \deg[X]{Y}:= \sum_{\substack{x<y \\x\in X\\y \in Y \\ x \notin Y}}Y(y) - 
                  \sum_{\substack{y<x \\x\in X\\y \in Y \\ x \notin Y}}Y(y). 
  \end{align*}
\end{dfn}

\begin{notation}
For $N$ a non-negative integer, denote by $\set{N}$ the ordered set of integers between $1$ and $N$.
\end{notation}

A version of the next proposition can be found in \cite{MR1865777} (Theorem 6.1  on page 19).

\begin{prop}
  Let $N$ and $k$ be two non-negative integers. The following identities hold:
  \begin{align*}
    \sum_{Y \in \ps[k]{\set{N}}} q^{\deg[\set{N}]{Y}} &=
    \begin{bmatrix}
      N \\k
    \end{bmatrix}
    \quad \textrm{and} \\
    (-1)^k
    \sum_{Y \in \mps[k]{\set{N}}} q^{\deg[\set{N}]{Y}} &=
    \begin{bmatrix}
      -N \\k
    \end{bmatrix}.
  \end{align*}
\end{prop}

\begin{proof}
  This is an easy induction. The statement is clear when $N=0$, and one can prove that the left-hand sides of these two identities satisfy the induction formulas given in Lemma~\ref{lem:pascal} and Corollary~\ref{cor:anti-pascal}.
\end{proof}

\section{Colorings of MOY graphs}
\label{sec:colorings-moy-graphs}
The aim of this section is to present a state sum à la Murakami--Ohtsuki--Yamada \cite{MR1659228} computing the Reshetikhin--Turaev invariants of links associated with the exterior powers of the standard representation of $U_q(\gll_{N|M})$. These invariants are denoted $P_{N|M}$. As in \cite{MR1659228}, this state sum is a combinatorial interpretation of the underlying representation theory. Details about the algebraic definition is provided in Appendix~\ref{sec:moy-graphs-an}. Note however that consistency of the MOY calculus presented here follows from its definition as a state sum and does not rely on the algebraic interpretation.
The $\sll_N$-invariants associated with exterior powers correspond to setting $M=0$.

Symmetric powers are covered by this state sum by essentially swaping $M$ and $N$. Indeed, 
the Hopf algebra $U_q(\gll_{M|N})$ is isomorphic to $U_{q}(\gll_{N|M})$ (see Remark~\ref{rmk:swapMN}). The standard representation $\CC^{M|N}_q$ of the first is isomorphic to the standard representation $\CC^{N|M}_{q}$ tensored with a trivial representation of odd super degree denoted $\boldsymbol{1}_{\mathrm{odd}}$. It induces isomorphisms between $\Sym_{q}^k\CC^{M|N}_q$ and $\Lambda^k\left(_{q} \CC^{N|M}_{q}\otimes \boldsymbol{1}_{\mathrm{odd}}\right)$. The later being isomorphic to
\begin{align*}
  \begin{cases}
    \Lambda^k_{q}\left(\CC^{N|M}_{q}\right) & \text{$k$ even,} \\[5pt]
    \Lambda^k_{q}\left(\CC^{N|M}_{q}\right) \otimes \boldsymbol{1}_{\mathrm{odd}} & \text{$k$ odd.}
  \end{cases}
\end{align*}

Since the $U_{q}(\gll_{N|M})$-invariant of a link with $\ell$ components associated with $\boldsymbol{1}_{\mathrm{odd}}$ is equal to $(-1)^\ell$, we obtain that if $L$ is an oriented link with $\ell$ components labeled by positive integers $k_1, \dots, k_\ell$, interpreted as exponents of symmetric powers, then its $U_q(\gll_{M|N})$-invariant is equal to \[(-1)^{\sum_ik_i}P_{N|M}(\overline{L}),\]
where $\overline{L}$ denote the mirror image of $L$ and the integers $k_1, \dots, k_\ell$ are now interpreted as exponents exterior powers.
The mirror image appears because the isomorphism between $U_q(\gll_{N|M})$ and $U_{q}(\gll_{M|N})$ maps the universal $R$-matrix of $U_q(\gll_{N|M})$ onto the inverse of that of $U_q(\gll_{M|N})$.

The most famous invariant associated with symmetric powers is the colored Jones polynomial of links. For a link $L$ with $\ell$ components, denote $J_n(L)$ the un-normalized\footnote{This means that the value of the unknot is $[n+1]$.} colored Jones polynomial associated with the $(n+1)$-dimensional irreducible representation of $U_q(\sll_2)$. The discussion above gives
\[P_{0|2}(L_{n, \dots,n}) = (-1)^{n\ell}J_n(\overline{L}) = J_n(L, -q^{-1}),\]
where $L_{n, \dots,n}$ means that every component of $L$ is labeled by $n$. For further reference, denote by $J'_n(L):= \frac{J_n(L)}{[n+1]}$ the normalized colored Jones polynomial.

\begin{dfn}[\cite{MR1659228}]
  \label{dfn:abstractMOYgraph}
  A \emph{MOY graph} is a finite oriented planar trivalent graph $\Gamma = (V(\Gamma),E(\Gamma))$ with a labeling of its edges $\ell : E_\Gamma \to \ZZ_{\geq0}$ such that the flow given by labels and orientations is preserved at each 
 vertex, meaning that every trivalent vertex follows one of the two models \[
\NB
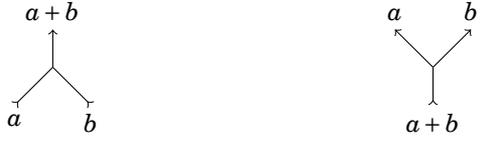
{\tikz{\begin{scope}
\draw[->] (0,0) -- (0,0.5) node [at end, above] {$a+b$};  
\draw[>-] (-0.5, -0.5) -- (0,0) node [at start, below] {$a$};  
\draw[>-] (+0.5, -0.5) -- (0,0) node [at start, below] {$b$};  
\begin{scope}[xshift = 5cm, yscale = -1]
  \draw[-<] (0,0) -- (0,0.5) node [at end, below] {$a+b$};  
\draw[<-] (-0.5, -0.5) -- (0,0) node [at start, above] {$a$};  
\draw[<-] (+0.5, -0.5) -- (0,0) node [at start, above] {$b$};  
\end{scope}
\end{scope} }}.
\]
The first is a \emph{merge} vertex, the second is a \emph{split} vertex. For each vertex, there is a \emph{thick} edge (the one labeled by $a+b$ in the depicted models) and two \emph{thin} edges: the one labeled by $a$ in the previous models is the \emph{left} thin edge, the one labeled by $b$ is the \emph{right} thin edge.
\end{dfn}

\begin{rmk}
Let $\Gamma$ be a MOY graph, then the graph $\Gamma'$ obtained by erasing the edges of $\Gamma$ with label $0$ is still a MOY graph. Moreover in all what follows the two graphs can be considered as the same. The label $0$ is introduced only to make some definitions and proofs easier and more natural.
\end{rmk}

\begin{dfn}
  \label{dfn:cabling-and-numbers}
  Let $\Gamma$ be a MOY graphs. The \emph{cabling} of $\Gamma$ is the oriented simple multi-curve in the plane $\RR^2$ obtained from $\Gamma$ by replacing each $e$ edge of $\Gamma$ by $\ell(e)$ parallel oriented intervals and joining these intervals at every vertex of $\Gamma$ by the only compatible oriented crossing-less matching.
  
Let $\gamma$ be a simple multi-curve in the plane. The \emph{rotational} of $\gamma$ is the integer $\rho(\gamma)$ equal to  the number of circles of $\gamma$ oriented counterclockwise minus the number of circles of $\gamma$ oriented clockwise. 

A circle oriented counterclockwise is said \emph{positively oriented} and a circle oriented clockwise is said \emph{negatively oriented}.

The \emph{rotational} of a MOY graph $\Gamma$  is the integer $\rho(\Gamma)$ equal to the rotational of its cabling.
\end{dfn}
\begin{figure}[ht]
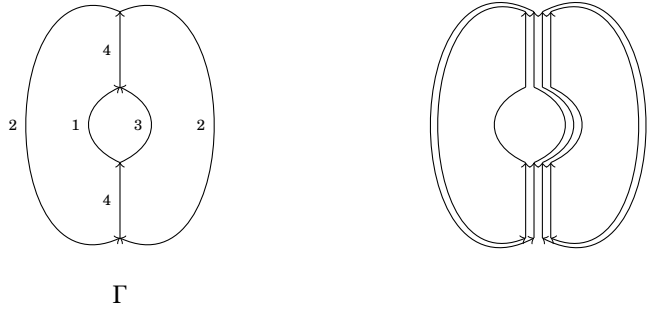

  \centering
  \tikz[yscale = 0.5, xscale=1.1]{\begin{scope}[font=\tiny]
  \coordinate (A) at (0,0);
  \coordinate (B) at (0,2);
  \coordinate (C) at (0,4);
  \coordinate (D) at (0,6);
  \draw[->] (A) -- (B) node[midway, left] {$4$};
  \draw[->] (C) -- (D) node[midway, left] {$4$};
  \draw[->] (B) .. controls +(0.5,0.5) and +(0.5,-0.5) .. (C) node[midway, left] {$3$};
  \draw[->] (B) .. controls +(-0.5,0.5) and +(-0.5,-0.5) .. (C) node[midway, left] {$1$};
  \draw[->] (D) .. controls +(1.5,1.5) and +(1.5,-1.5) .. (A) node[midway, left] {$2$};
  \draw[->] (D) .. controls +(-1.5,1.5) and +(-1.5,-1.5) .. (A) node[midway, left] {$2$};
\end{scope}
  \node at (0, -1.5) {$\Gamma$};

\begin{scope}[xshift = 5cm]
  \coordinate (A1) at (-0.15,0);
  \coordinate (A2) at (-0.05,0);
  \coordinate (A4) at (+0.15,0);
  \coordinate (A3) at (+0.05,0);
  \coordinate (B1) at (-0.15,2);
  \coordinate (B2) at (-0.05,2);
  \coordinate (B4) at (+0.15,2);
  \coordinate (B3) at (+0.05,2);
  \coordinate (C1) at (-0.15,4);
  \coordinate (C2) at (-0.05,4);
  \coordinate (C4) at (+0.15,4);
  \coordinate (C3) at (+0.05,4);
  \coordinate (D1) at (-0.15,6);
  \coordinate (D2) at (-0.05,6);
  \coordinate (D4) at (+0.15,6);
  \coordinate (D3) at (+0.05,6);
  \draw[->] (A1) -- (B1);
  \draw[->] (A2) -- (B2);
  \draw[->] (A3) -- (B3);
  \draw[->] (A4) -- (B4);
  \draw[->] (C1) -- (D1);
  \draw[->] (C2) -- (D2);
  \draw[->] (C3) -- (D3);
  \draw[->] (C4) -- (D4);
  \draw (B1) .. controls +(-0.5,0.5) and +(-0.5,-0.5) .. (C1);
  \draw (B2) .. controls +(0.5,0.5) and +(0.5,-0.5) .. (C2);
  \draw (B3) .. controls +(0.5,0.5) and +(0.5,-0.5) .. (C3);
  \draw (B4) .. controls +(0.5,0.5) and +(0.5,-0.5) .. (C4); 
   \draw[->] (D3) .. controls +(1.65,1.9) and +(1.65,-1.9) .. (A3);
   \draw[->] (D4) .. controls +(1.4,1.3) and +(1.4,-1.3) .. (A4);
   \draw[->] (D1) .. controls +(-1.4,1.3) and +(-1.4,-1.3) .. (A1);
   \draw[->] (D2) .. controls +(-1.65,1.9) and +(-1.65,-1.9) .. (A2);
\end{scope} }
  \caption{A graph MOY $\Gamma$ and its cabling. Its rotational $\rho(\Gamma)$ is equal to $2-2 =0$.}
  \label{fig:cabling}
\end{figure}

\begin{dfn}\label{dfn:sym1}
  The \emph{binomial weight} of a MOY graph $\Gamma$ is the Laurent polynomial $b(\Gamma)$ defined by the following formula:
  \[
  b(\Gamma) = \prod_{v\text{ split vertex}}
  \begin{bmatrix}
    v_l + v_r \\ v_r
  \end{bmatrix},
\]
where for each split vertex $v$, $v_l$ (\resp{}$v_r$) denotes the label of the left thin edge (\resp{}right thin edge) of $v$.
\end{dfn}

\begin{dfn}
  \label{dfn:subMOY}
  Let $\Gamma= (V(\Gamma), E(\Gamma), \ell)$ be a MOY graph. 
A \emph{sub-MOY graph of $\Gamma$} is a MOY graph $\Gamma'=(V(\Gamma), E(\Gamma), \ell') $ such that 
for all $e$ in $E(\Gamma)$, $\ell'(e) \leq \ell(e)$. 

A sub-MOY graph $\Gamma'$ is \emph{cyclic} if for all $e \in E(\Gamma)$, $\ell'(e) \leq 1$.

Let $I$ be a finite set and consider a collection $\Gamma_i= \left(V(\Gamma), E(\Gamma), \ell_i)\right)_{i\in I}$ of sub-MOY graphs of $\Gamma$, the \emph{sum} of the sub-MOY graphs $\left(\Gamma_i\right)_{i\in I}$ is the MOY graph
$(V(\Gamma), E(\Gamma), \sum_{i\in I} \ell_i)$ and is denoted by $\sum_{i \in I}\Gamma_i$.  
\end{dfn}

\begin{dfn}
  \label{dfn:u}
  Let $\Gamma= (V(\Gamma), E(\Gamma), \ell)$ be a MOY graph and $\Gamma_1= \left(V(\Gamma), E(\Gamma), \ell_1)\right)$ and $\Gamma_2= \left(V(\Gamma), E(\Gamma), \ell_2)\right)$ be two sub-MOY graphs of $\Gamma$. For each split vertex $s$ of $\Gamma$, define $u_s(\Gamma_1, \Gamma_2)$ to be the integer defined by the formula:
\[
u_s(\Gamma_1, \Gamma_2):= a_2b_1 - a_1b_2,
\]
where $a_1$ (\resp{}$b_1$) is the label of $\Gamma_1$ of the left (\resp{}right) small edge at the vertex $s$ and $a_2$ (\resp{}$b_2$) is the label of $\Gamma_2$ of the left (\resp{}right) small edge at the vertex $s$. Define   
$u(\Gamma_1, \Gamma_2)$ to be the integer given by the formula:
\[
u(\Gamma_1, \Gamma_2):= \sum_{\textrm{$s$ split vertex of $\Gamma$}} u_s(\Gamma_1, \Gamma_2).
\]
\end{dfn}

\begin{rmk}
Note that the sum of sub-MOY graphs of $\Gamma$ is a MOY graph but not necessarily a sub-MOY graph of $\Gamma$, since it is not required that for all $e\in E(\Gamma)$, $\sum_{i\in I} \ell_i(e)\leq \ell(e)$.
\end{rmk}

For the rest of the paper, fix two non-negative integers $N$ and $M$.

\begin{dfn}
  \label{dfn:coloring}
  Let $\Gamma$ be a MOY graph, an \emph{$(N|M)$-coloring} (or simply \emph{coloring}) of $\Gamma$ is a pair $(\underline{\Gamma^E}, \underline{\Gamma^S})$ where  $\underline{\Gamma^E} = (\Gamma^E_1, \dots, \Gamma^E_N)$ is an $N$-tuple of cyclic sub-MOY graph of $\Gamma$ and $\underline{\Gamma^S} = (\Gamma^S_1, \dots, \Gamma^S_M)$ is an $M$-tuple of sub-MOY graph of $\Gamma$ such that:
\[
\sum_{j=1}^N \Gamma^E_{j} 
+
\sum_{i=1}^M \Gamma^S_{i} 
=
\Gamma.
\] 
The set of $(N|M)$-colorings of $\Gamma$ is denoted by $\col[N|M]{\Gamma}$ (or simply $\col\Gamma$).
\end{dfn}

\begin{rmk}
  \label{rmk:col2multiset}
  Let $\Gamma$ be a MOY graph and $c = (\underline{\Gamma^E}, \underline{\Gamma^S})$ a coloring of $\Gamma$. If $e$ is an edge of $\Gamma$, consider the pair $(c^E(e), c^S(e))$:  $c^E(e)$ is the subset of $\set{N}$ which contains $j$ if and only if $\ell_{\Gamma^E_j}(e)=1$ and $c^S(e)$ is the sub-multiset of $\set{M}$ where $i$ appears with multiplicity $\ell_{\Gamma^S_i}(e)$. We have $\#c^E(e) + \#c^S(e) = \ell(e)$. A coloring can be defined by such a  collection  $(c^E(e), c^S(e))_{e\in E(\Gamma)}$ satisfying a flow condition at every vertex.
\end{rmk}

\begin{dfn}
  \label{dfn:weight}
  Let $\Gamma$ be a MOY graph and $c=(\underline{\Gamma^E}, \underline{\Gamma^S})$ an $(N|M)$-coloring of $\Gamma$. The \emph{$s$-weight of $c$} is the integer denoted by $w_s(c)$ and defined by the following formula:
\begin{align*}
 w_s(c):=& \sum_{1\leq i <j \leq M} u(\Gamma^S_j,\Gamma_i^S) -\sum_{1\leq i <j \leq N} u(\Gamma^E_j,\Gamma_i^E) + \sum_{i=1}^N\sum_{j=1}^M u(\Gamma^E_i,\Gamma_j^S).
\end{align*}
The \emph{$\rho$-weight of $c$} is the integer denoted by $w_\rho(c)$ and defined by the following formula:

\begin{align*}
 w_\rho(c):=& \sum_{i=1}^N(N+M-2i+1) \rho(\Gamma^E_i) + \sum_{j=1}^M (N+M-2j+1)\rho(\Gamma^S_j). 
 \end{align*}
The \emph{multiplicity of $c$} is the Laurent polynomial $m(c)$ with positive coefficients and symmetric in $q$ and $q^{-1}$ given by the following formula:
\[
m(c) = \prod_{i=1}^M b(\Gamma_i^S).
\]
The \emph{parity of $c$} is the integer $s(c)$ given by the following formula:
\[
  s(c) = {\sum_{i=1}^M \rho(\Gamma_i^S)},
\]
as the name suggests, we are only interested in the value of $s(c)$ modulo $2$. Equality modulo $2$ is denoted by $\equiv$ in what follows.
The \emph{combinatorial $(N|M)$-evaluation of the coloring $c$} is the Laurent polynomial $\kup{\Gamma,c}_{N|M}$ given by the following formula:
\[
\kup{\Gamma,c}_{N|M} =  (-1)^{s(c)} q^{w_s(c)+w_\rho(c)}m(c).
\]
Finally, the \emph{combinatorial $(N|M)$-evaluation of $\Gamma$} is the Laurent polynomial $\kup{\Gamma}_{N|M}$ given by the following formula:
\[
\kup{\Gamma}_{N|M} = \sum_{c\in \col{\Gamma}} \kup{\Gamma,c}_{N|M}.
\]
\end{dfn}

We give in Appendix A a detailed account of the Reshetikhin--Turaev invariant for MOY graphs for the quantum group $U_q(\gll_{N|M})$. For any MOY graph $\Gamma$ it is denoted $\kups{\Gamma}_{N|M}$. The aim of the rest of this section is to prove the next theorem.

\begin{thm}
  \label{thm:maintheorem}
  For any MOY graph $\Gamma$, the algebraic and the combinatorial evaluation of $\Gamma$ agree. This means:
\[
\kups{\Gamma}_{N|M} = \kup{\Gamma}_{N|M}.
\]
In particular, the combinatorial evaluation depends only on $N-M$ and is symmetric in $q$ and $q^{-1}$. 
\end{thm}

Following Proposition \ref{prop:completeness}, one only needs to check that $\kup{\bullet}_{N|M}$ satisfies the multiplicativity property and the local relations of Proposition \ref{prop:rel-kups}. This is the object of the rest of this section. 

\begin{lem}
  \label{prop:mirror-image}
  Let $\Gamma$ be a MOY graph, and $\overline{\Gamma}$ be its image under the transformation
  $\left(\begin{smallmatrix}
    x \\ y
  \end{smallmatrix}\right)\mapsto \left(\begin{smallmatrix}
    -x \\ y
  \end{smallmatrix}\right)$ Then
  $\kup{\Gamma}_{N|M}(q) = \kup{\overline{\Gamma}}_{N|M}(q^{-1})$.
\end{lem}
\begin{proof}
  There is a canonical one-to-one correspondence between colorings of $\Gamma$ and colorings of $\overline{\Gamma}$. Let $c$ be a coloring of $\Gamma$ and $\overline{c}$ the corresponding coloring of $\overline{\Gamma}$. The following identities hold:
  \begin{align*}
    m(c)(q) &= m(c)\left(q^{-1}\right) = m\left(\overline{c}\right)(q) = m\left(\overline{c}\right)\left(q^{-1}\right),\\
s(c) &\equiv s\left(\overline{c}\right), \\
w_s(c) &= - w_s\left(\overline{c}\right) 
\quad\textrm{and} \\
w_\rho(c) &= -w_\rho\left(\overline{c}\right).
\end{align*}
From this, one immediately obtains $\kup{\Gamma}_{N|M}(q) = \kup{\overline{\Gamma}}_{N|M}(q^{-1})$.

\end{proof}

\begin{lem}
  \label{lem:disjointunion}
  Let $\Gamma$ and $\Upsilon$ be two MOY graphs, then
\[
\kup{\Gamma \sqcup \Upsilon}_{N|M} = \kup{\Gamma}_{N|M}\kup{\Upsilon}_{N|M}.
\]
\end{lem}

\begin{proof}
  First of all, there is a canonical one-to-one correspondence between the set of colorings of $\Gamma \sqcup \Upsilon$ and the Cartesian product of sets of colorings of $\Gamma$ and $\Upsilon$. If $c$ and $d$ are colorings of $\Gamma$ and $\Upsilon$, denote by $(c,d)$ the corresponding coloring of $\Gamma\sqcup\Upsilon$. It is enough to prove that for all colorings $(c,d)$,
\[
(-1)^{s(c)+s(d)} q^{w_s(c)+w_\rho(d)+w_s(c) + w_\rho(d)}m(c)m(d) 
=
(-1)^{s((c,d))} q^{w_s((c,d)) + w_\rho((c,d))}m((c,d)).
\]
Let $c=(\underline{\Gamma^E}, \underline{\Gamma^S})$ (\resp{}$d=(\underline{\Upsilon^E}, \underline{\Upsilon^S})$) be a coloring of $\Gamma$ (\resp{}$\Upsilon$). 
One has $s((c,d))=s(c)+s(d)$
and $m((c,d)) = m(c)m(d).$ Hence, it remains to prove that $w((c,d)) = w(c) + w(d)$. Let us denote by $\Phi$ the MOY graph $\Gamma \sqcup \Upsilon$. 
One has
\begin{align*}
  w_s((c,d))&= \sum_{1\leq i <j \leq M} u(\Phi^S_j,\Phi_i^S) -\sum_{1\leq i <j \leq N} u(\Phi^E_j,\Phi_i^E) + \sum_{i=1}^N\sum_{j=1}^M u(\Phi^E_i,\Phi_j^S) \\
&=  \sum_{1\leq i <j \leq M} u(\Gamma^S_j,\Gamma_i^S) + u(\Upsilon^S_j,\Upsilon_i^S)  -\sum_{1\leq i <j \leq N} u(\Gamma^E_j,\Gamma_i^E) + u(\Upsilon^E_j,\Upsilon_i^E) \\ &
+ \sum_{i=1}^N\sum_{j=1}^M u(\Gamma^E_i,\Gamma_j^S) +u(\Upsilon^E_i,\Upsilon_j^S). \\
=& w_s(c) + w_s(d)
\end{align*}
and
\begin{align*}
  w_\rho((c,d))=&
\sum_{i=1}^M(N+M-2i+1)\rho(\Phi^S_i)  +
\sum_{i=1}^N(N+M-2i+1)\rho(\Phi^E_i)  \\
= & \sum_{i=1}^M (N+M-2i+1)\left(\rho(\Gamma^S_i) + \rho(\Upsilon^S_i)\right) \\ & +  \sum_{i=1}^N (N+M-2i+1)\left(\rho(\Gamma^E_i) + \rho(\Upsilon^E_i)\right) \\
=& w_\rho(c) + w_\rho(d). \qedhere
\end{align*}
\end{proof}
\begin{lem}
  \label{lem:circles}
  For any non-negative integer $k$, the following identity holds: 
  \[
\kup{\NB{\tikz{\draw[->] (0,0) arc(0:360:0.4) node[scale=0.6, right] {$k$};}}\!}_{N|M} 
=
\begin{bmatrix}
  N-M \\ k
\end{bmatrix}
=
\kup{\NB{\tikz{\draw[<-] (0,0) arc(0:360:0.4) node[scale=0.6, right] {$k$};}}\!}_{N|M}.
\]
\end{lem}
\begin{proof}
Let us start with the first identity. Denote by $\Gamma$ the positively oriented circle of label $k$. A  coloring $c=(\underline{\Gamma^E}, \underline{\Gamma^S})$  of $\Gamma$ is completely encoded by a pair consisting of a subset $X_c$ of $\set{N}$ with $k^E_c$ elements a multi-subset $Y_c$ of $\set{M}$ with $k^S_c$ elements such that $k_c^E+ k_c^S =k$. 
One has $s(c) = k^S_c$ and $m(c)= 1$. Since there is no vertex, $w_s(\Gamma)=0$.
One has
\begin{align*}
  w_\rho(c) =&  \sum_{i=1}^M (N+M-2i+1)\rho(\Gamma^S_i) + \sum_{i=1}^N(N+M-2i+1) \rho(\Gamma^E_i) \\
    =&  Nk^S_c + Mk^E_c +  \sum_{1\leq i <j \leq M} \rho (\Gamma^S_j) - \rho(\Gamma_i^S) + \sum_{1\leq i <j \leq N} \rho (\Gamma^E_j) - \rho(\Gamma_i^E).  
\end{align*}
Hence, if one fixes two integers $k^S$ and $k^E$ such that $k= k^S + k^E$, one gets:
\begin{align*}
  \sum_{\substack{c \in \col{\Gamma} \\ k^S_c = k^S \\ k^E_c= k^E}}
  q^{w_\rho(c)}&=  
q^{Mk^E + Nk^S} \!\!\!\!\!
\sum_{\substack{(\underline{\Gamma^E}, \underline{\Gamma^S}) \in \col{\Gamma} 
\\ k^S_c = k^S \\ k^E_c= k^E}}  \!\!\!\!\!
q^{\sum_{1\leq i <j \leq M} \rho (\Gamma^S_j) - \rho(\Gamma_i^S)+\sum_{1\leq i <j \leq N} \rho (\Gamma^E_j) - \rho(\Gamma_i^E)}\\
&=
q^{Mk^E + Nk^S} 
\sum_{\substack{ X \in \ps[k^E]{\set{N}}\\Y\in \mps[k^S]{\set{M}}}}
q^{\deg[\set{M}]{X}} q^{\deg[\set{N}]{Y}} \\
&= q^{Mk^E_c + Nk^S_c} 
\begin{bmatrix}
  M + k^S -1 \\ k^S 
\end{bmatrix}
\begin{bmatrix}
  N \\ k^E
\end{bmatrix}.
\end{align*}
Hence, using Proposition~\ref{prop:dec-binomial}, one gets:
\begin{align*}
  \kup{\Gamma}_{N|M}&= \sum_{k^S + k^E= k} (-1)^{k^S} q^{Mk^E + Nk^S} 
\begin{bmatrix}
  M + k^S -1 \\ k^S 
\end{bmatrix}
\begin{bmatrix}
  N \\ k^E
\end{bmatrix}\\
&=\sum_{k^S + k^E= k} q^{Mk^E + Nk^S} 
\begin{bmatrix}
  - M \\ k^S 
\end{bmatrix}
\begin{bmatrix}
  N \\ k^E
\end{bmatrix} \\ &
= \begin{bmatrix}
  N - M \\k
\end{bmatrix}.
\end{align*}
The second identity follows from the first one and Lemma~\ref{prop:mirror-image} (because quantum binomials are symmetric in $q$ and $q^{-1}$).
\end{proof}
\begin{lem}
  \label{lem:assoc}
  The combinatorial $(N|M)$-evaluation satisfies the following two  local identities:
\begin{align} \label{eq:extrelass}
   \kup{\stgamma}_{N|M} &= \kup{\stgammaprime}_{N|M},
\\   \kup{\stgammar}_{N|M} &= \kup{\stgammaprimer}_{N|M}.
 \end{align}
\end{lem}
\begin{proof}
We only prove the first one. Let us denote by $\Gamma$  and $\Upsilon$ respectively the MOY graph on the left and right-hand side of this identity. There is a canonical one-to-one correspondence between the $(N|M)$-colorings of $\Gamma$ and of $\Upsilon$. Let $c=(\underline{\Gamma^E}, \underline{\Gamma^S})$ be a $(N|M)$-coloring of $\Gamma$ and $c'=(\underline{\Upsilon^E}, \underline{\Upsilon^S})$ be the corresponding coloring of $\Upsilon$. The following holds:
\[
m(c) = m(c'), \quad
s(c) = s(c'),\quad
w_s(c) = w_s(c')\quad
\textrm{and} \quad
w_\rho(c) = w_\rho(c').
\]
This implies $\kup{\Gamma,c}_{N|M} = \kup{\Upsilon, c'}_{N|M}$ and subsequently $\kup{\Gamma}_{N|M} = \kup{\Upsilon}_{N|M}$.
\end{proof}

\begin{lem}
  \label{lem:easy-digon}
  The combinatorial $(N|M)$-evaluation satisfies the following identity:
\begin{align} \label{eq:easy-digon}
   \kup{\digonaone}_{N|M} = [m+1]\kup{\,\vertaone\!}_{N|M}.
   \end{align}
\end{lem}

\begin{proof}
Let $\Gamma$ and $\Upsilon$ be two MOY graphs which are identical except in a small ball where they are related by the local relation (\ref{eq:easy-digon}), $\Gamma$ being on the left-hand side and $\Upsilon$ on the right-hand side.
  First note that any coloring of $\Gamma$ induces a coloring of $\Upsilon$. Let us fix a coloring $c'= (\underline{\Upsilon^E}, \underline{\Upsilon^S})$ of $\Upsilon$. If a coloring $c$ of $\Gamma$ induces $c'$ on $\Upsilon$, we write $c\to c'$. We will prove the following:
  \[
\sum_{\substack{c\to c' \\ c \in \col{\Gamma}}} \kup{\Gamma, c}_{N|M} = [m+1] \kup{\Upsilon, c'}_{N|M},
  \]
  which implies $\kup{\Gamma}_{N|M} = [m+1] \kup{\Upsilon}_{N|M}$.
  Denote by $e$ the edge of $\Upsilon$ which appears in the local relation. For $i$ in $\set{M}$ and $j$ in $\set{N}$, set
  \begin{align*}
    k^S_i := \ell_{\Upsilon_i^S}(e), \quad k^E_j := \ell_{\Upsilon_j^E}(e), \quad
    k^S := \sum_{i=1}^M k^S_i, \quad \textrm{and} \quad k^E := \sum_{j=1}^N k^E_j.
  \end{align*}

  A coloring $c= (\underline{\Gamma^E}, \underline{\Gamma^S})$ of $\Gamma$ inducing $c'$ is totally determined by saying which MOY graph of the collection $\underline{\Gamma^E}\sqcup \underline{\Gamma^S}$ gives label $1$ to the right edge of the digon which appears in the local relation. Denote this graph by $\Gamma^A_h$. There are two possibilities: either this graph belongs to $\underline{\Gamma^S}$ (\ie{}$A=S$ and $h\in \set{M}$) or to $\underline{\Gamma^E}$ (\ie{}$A=E$ and $h\in \set{N}$).
  If $A = S$ and  $h\in \set{M}$, one has:
  \begin{align*}
    m(c) &= m(c')[k_h^S], &&\quad& s(c) &= s(c'),\\
    w_\rho(c)&= w_\rho(c')&& \quad \textrm{and}\quad & 
    w_s(c) &= w_s (c') - \sum_{i=1}^{h-1} k_i^S + \sum_{i=h+1}^M k_i^S - k^E.
  \end{align*}
  If $A = E$ and $h \in \set{N}$, one has:
  \begin{align*}
    m(c) &= m(c')[k_h^E], 
    \\
    s(c) &= s(c'),\\
    w_\rho(c)&= w_\rho(c') \quad \textrm{and}\\
    w_s(c) &= w_s (c') + \sum_{i=1}^{h-1} k_i^E - \sum_{i=h+1}^N k_i^E + k^S.
  \end{align*}
  Note that for the coloring $c$ to exist, it is necessary that  $k_h^E=1$. However, if $k_h^E=0$, the previous formula makes the contribution of this ``virtual'' coloring equal to  $0$.
  This argument will be used throughout the proofs of this section. We will not repeat it.
  One can sum:
  \begin{align*}
    \sum_{\substack{c\to c' \\ c\in \col{\Gamma} }} \kup{\Gamma, c}_{N|M}
    &= \left(\sum_{h=1}^M q^{-k^E- \sum_{i=1}^{h-1} k_i^S + \sum_{i=h+1}^M k_i^S}[k_h^S] \right. \\
    &\hspace{2cm} \left. + \sum_{h=1}^N
      q^{k^S + \sum_{i=1}^{h-1} k_i^E - \sum_{i=h+1}^N k_i^E}[k_h^E] \right) \kup{\Upsilon, c'}_{N|M} \\
&= \left(q^{-k^E}[k^S] + q^{k^S}[k^E] \right)\kup{\Upsilon, c'}_{N|M} \\ &= [k^S+k^E]\kup{\Upsilon, c'}_{N|M} = [m+1]\kup{\Upsilon, c'}_{N|M}. \qedhere
  \end{align*}
\end{proof}
\begin{cor} \label{cor:easy-digon2}
  The combinatorial $(N|M)$-evaluation satisfies the following local identities and their mirror images:
  \begin{align*}
    \kup{\digonbone}_{N|M}&=  [m+1]\kup{\,\vertaone\!}_{N|M}, \\
    \kup{\digona}_{N|M}&=  \begin{bmatrix}  m+n \\ m \end{bmatrix}
    \kup{\,\verta\!}_{N|M},\\
    \kup{\squaree}_{N|M} &=  \begin{bmatrix}  r+s \\ r \end{bmatrix}
    \kup{\webHe}_{N|M}.
    \end{align*} 
\end{cor}

\begin{proof}
  The first identity is a consequence of Lemma~\ref{lem:easy-digon} and Lemma~\ref{prop:mirror-image}. The second is an easy induction. The third is a consequence of the second and of Lemma~\ref{lem:assoc}.
\end{proof}

\begin{lem}
  \label{lem:bad-digon}
  The combinatorial $(N|M)$-evaluation satisfies the following local identity:
\begin{align} \label{eq:bad-digon}
   \kup{\digonbb}_{N|M} = [N-M-m]\kup{\,\vertb\!}_{N|M}.
 \end{align}
\end{lem}
  
\begin{proof}
Let $\Gamma$ and $\Upsilon$ be two MOY graphs which are identical except in a small ball where they are related by the local relation (\ref{eq:bad-digon}), $\Gamma$ being on the left-hand side and $\Upsilon$ on the right-hand side.
  First note that any coloring of $\Gamma$ induces a coloring of $\Upsilon$. Fix a coloring $c'= (\underline{\Upsilon^E}, \underline{\Upsilon^S})$ of $\Upsilon$. If a coloring $c$ of $\Gamma$ induces $c'$ on $\Upsilon$, we write $c\to c'$. We will prove the following:
  \[
\sum_{\substack{c\to c'\\c \in \col{\Gamma} }} \kup{\Gamma, c}_{N|M} = [N-M-m] \kup{\Upsilon, c'}_{N|M},
  \]
  which implies $\kup{\Gamma}_{N|M} = [N-M-m] \kup{\Upsilon}_{N|M}$.
  Denote by $e$ the edge of $\Upsilon$ which appears in the local relation. For $i$ in $\set{M}$ and $j$ in $\set{N}$, set
  \begin{align*}
    k^S_i := \ell_{\Upsilon_i^S}(e), \quad k^E_j := \ell_{\Upsilon_j^E}(e), \quad
    k^S := \sum_{i=1}^M k^S_i, \quad \textrm{and} \quad k^E := \sum_{i=1}^N k^E_i.
  \end{align*}

  A coloring $c= (\underline{\Gamma^S}, \underline{\Gamma^S})$ of $\Gamma$ inducing $c'$ is totally determined by saying which MOY graph of the collection $\underline{\Gamma^E}\sqcup \underline{\Gamma^S}$ gives label $1$ to the right edge of the digon which appears in the local relation. Denote this graph by $\Gamma^A_h$. There are two possibilities: either this graph belongs to $\underline{\Gamma^S}$ (\ie{}$A=S$ and $h\in \set{M}$) or to $\underline{\Gamma^E}$ (\ie{}$A=E$ and $h\in \set{N}$).
  If $A = S$ and  $h\in \set{M}$, one has:
  \begin{align*}
    m(c) &= m(c')[k_h^S+1], \\
    s(c) &= s(c') +1,\\
w_\rho(c) &= w_\rho(c') + N+M-2h +1 \quad \textrm{and}\\
    w_s(c) &= w_s(c')- \sum_{i=1}^{h-1} k_i^S + \sum_{i=h+1}^M k_i^S - k^E .
  \end{align*}
Hence:
\begin{align*}  
w_\rho(c) + w_s(c)&= w_\rho(c') + w_s(c') + N - k^E- \sum_{i=1}^{h-1} (k_i^S +1) + \sum_{i=h+1}^M (k_i^S + 1).
  \end{align*}
  If $A = E$ and $h \in \set{N}$, one has: 
  \begin{align*}
    m(c) &= m(c')[1-k_h^E] \quad \textrm{note that one has $k^E_h=0$ }, \\
    s(c) &= s(c'),\\
w_\rho(c)& = w_\rho(c') + N+M-2h +1 \quad \textrm{and}\\
    w_s(c) &= w_s(c')+ \sum_{i=1}^{h-1} k_i^E - \sum_{i=h+1}^N k_i^E + k^S .\\
  \end{align*}
Hence:
\begin{align*}  
  w_\rho(c) + w_s(c)&= w_\rho(c') + w_s(c') + M + k^S- \sum_{i=1}^{h-1} (1-k_i^E) + \sum_{i=h+1}^N (1-k_i^E ).
  \end{align*}
Finally, one gets:
  \begin{align*}
    \sum_{\substack{c\to c'\\ c\in \col{\Gamma}}} \kup{\Gamma, c}_{N|M} =&\left(
-\sum_{h=1}^M q^{N -k^E - \sum_{i=1}^{h-1} (k_i^S +1) + \sum_{i=h+1}^M (k_i^S + 1)}
[k_h^S+1] \right. \\ & \hspace{1cm}+ \left.
\sum_{h=1}^N q^{M + k^S -\sum_{i=1}^{h-1} (1-k_i^E) + \sum_{i=h+1}^N (1-k_i^E)}
[1-k_h^E] 
\right)
\kup{\Upsilon, c'}
\\ =& \left(-q^{N-k^E}[M+k^S] + q^{M+k^S}[N-k^E] \right)\kup{\Upsilon, c'} \\ =& [N-M -m]\kup{\Upsilon, c'} .
  \end{align*}
\end{proof}

\begin{cor} \label{cor:bad-digon2}
  The combinatorial $(N|M)$-evaluation satisfies the following local identities:
  \begin{align*}
    \kup{\digonbbb}_{N|M}&=  [N-M-m]\kup{\,\vertb\!}_{N|M}, \\
    \kup{\digonb}_{N|M}&=  \begin{bmatrix}  N-M-m \\ n \end{bmatrix}
      \kup{\,\vertb\!}_{N|M} = \kup{\digonbbbb}_{N|M}.
    \end{align*} 
\end{cor}
\begin{proof}
  The first identity is a consequence of Lemma~\ref{lem:bad-digon} and Lemma~\ref{prop:mirror-image}. The second is an easy induction.
\end{proof}

\begin{lem}
  \label{lem:easy-square}
    The combinatorial $(N|M)$-evaluation satisfies the following local identity:
\begin{align} \label{eq:easy-square}
  \kup{\squareccc}_{N|M} = [k]\kup{\twovertddd}_{N|M}
  +  \kup{\squareddd}_{N|M}.
 \end{align}
\end{lem}

\begin{proof}
  Let $\Gamma$, $\Upsilon$ and $\Phi$ be three MOY graphs which are identical except in a small ball $B$ where they are related by the local relation (\ref{eq:easy-square}), $\Gamma$ being on the left-hand side and $\Upsilon$ and $\Phi$ on the right-hand side ($\Upsilon$ the first and $\Phi$ the second).
  Let $c$ be a coloring of $\Gamma$. There are two possibilities: either the two rungs receive the same color or they receive different colors.
  First consider a coloring $c=(\underline{\Gamma^E}, \underline{\Gamma^S})$ which gives different colors to the two rungs.
  There is no coloring of $\Upsilon$ which coincides with $c$ outside $B$. On the other hand, it induces a unique coloring $c''=(\underline{\Phi^E}, \underline{\Phi^S})$ of $\Phi$ which associates different colors to the two rungs of $\Phi$. One has:
   \begin{align*}
    m(c) &= m(c''), && \quad\quad & 
    s(c) &= s(c''),\\
w_\rho(c) &= w_\rho(c'') &&\quad \textrm{and}& 
    w_s(c) &= w_s(c'').
  \end{align*}
  Hence
  \begin{align}\label{eq:easysquare-firstcase}
\kup{\Gamma, c}_{N|M} = \kup{\Phi,c''}_{N|M} 
\end{align}

   Consider now the case where $c$ gives the same color to the two rungs. It induces a coloring $c'$ of $\Upsilon$. Similarly there are colorings of $\Phi$ which induce $c'$ on $\Upsilon$. Let us write $c\to c'$ to indicate that $c$ induces the coloring $c'$ on $\Upsilon$. We will show the following:
  \begin{align*}
&    \sum_{\substack{c\in \col{\Gamma} \\ c\to c'}}
\kup{\Gamma, c}_{N|M}
- 
    \sum_{\substack{c\in \col{\Phi} \\ c\to c'}}
\kup{\Phi, c''}_{N|M}
= [k] \kup{\Upsilon, c'}_{N|M}
  \end{align*}
  which together with (\ref{eq:easysquare-firstcase}) implies the lemma. 
  
  Fix $c'= (\underline{\Upsilon^E}, \underline{\Upsilon^S})$ a coloring of $\Upsilon$. Let us denote by $e_l$ (\resp{}$e_r$) the vertical edge on the left (\resp{}right) of the part of $\Upsilon$ which is in $B$. We need a few extra notations. For $i \in \set{M}$ and $j \in \set{N}$, set:
  \begin{align*}
    l_i^S &:= \ell_{\Gamma_i^S}(e_l), &\quad \quad&& r_i^S &:= \ell_{\Gamma_i^S}(e_r),
    &\quad \quad&& l_j^E &:= \ell_{\Gamma_j^E}(e_l), &\quad \quad&& r_j^E &:= \ell_{\Gamma_j^E}(e_r), \\
    l^S &:= \sum_{i=1}^Ml_i^S, &\quad \quad&& r^S &:= \sum_{i=1}^Mr_i^S, 
&\quad \quad&&    l^E &:= \sum_{j=1}^Nl_j^E, &\quad \quad&&
r^E &:= \sum_{j=1}^Nr_j^E.
  \end{align*}
Note that $k=r^E+ r^S- l^E-l^S$.
  A coloring of $\Gamma$ which induces $c'$ is totally determined by the color it gives to the two rungs. Let $c= (\underline{\Gamma^E}, \underline{\Gamma^S})$ and  denote by $\Gamma^A_h$ the sub-MOY graph of $\Gamma$ which has label $1$ on the two rungs. There are two possibilities: either $A=S$ and $h \in \set{M}$ or $A=E$ and $h \in \set{N}$.
  In the case $A=S$ and $h \in \set{M}$, one has:
   \begin{align*}
    m(c) &= m(c') [l^S_h+1][r_h ^S] \\
    s(c) &= s(c'),\\
w_\rho(c) &= w_\rho(c') \quad \textrm{and}\\
    w_s(c) &= w_s(c') + r^E - l^E + \sum_{i=1}^{h-1} (l^S_i - r^S_i) - \sum_{i=h+1}^{M} (l^S_i - r^S_i).
  \end{align*}

  In the case $A=E$ and $h \in \set{N}$, one has:
   \begin{align*}
    m(c) &= m(c') [1-l^E_h][r_h ^E] \\
    s(c) &= s(c'),\\
w_\rho(c) &= w_\rho(c') \quad \textrm{and}\\
    w_s(c) &= w_s(c') + l^S - r^S + \sum_{i=1}^{h-1} (r^E_i - l^E_i) - \sum_{i=h+1}^{N} (r^E_i - l^E_i).
  \end{align*}
  Note that in order $c$ to exist it is necessary and sufficient that $r_h^E=1$ and $l^E_h=0$. 

  We perform the same analysis with a coloring $c'' = (\underline{\Phi^E}, \underline{\Phi^S})$ inducing $c'$ on $\Upsilon$. Denote by $\Phi^A_h$ the sub-MOY graph of $\Phi$ which has label $1$ on the two rungs. There are two possibilities: either $A=S$ and $h \in \set{M}$ or $A=E$ and $h \in \set{N}$.
  In the case $A=S$ and $h \in \set{M}$, one has:
     \begin{align*}
    m(c'') &= m(c') [l^S_h][r_h ^S+1] \\
    s(c'') &= s(c'),\\
w_\rho(c'') &= w_\rho(c') \quad \textrm{and}\\
    w_s(c'') &= w_s(c') + r^E - l^E + \sum_{i=1}^{h-1} (l^S_i - r^S_i) - \sum_{i=h+1}^{M} (l^S_i - r^S_i).
  \end{align*}
In the case $A=E$ and $h \in \set{N}$, one has:
   \begin{align*}
    m(c'') &= m(c') [l^E_h][1-r_h ^S] \\
    s(c'') &= s(c'),\\
w_\rho(c'') &= w_\rho(c') \quad \textrm{and}\\
    w_s(c'') &= w_s(c') + l_S - r_S + \sum_{i=1}^{h-1} (r^E_i - l^E_i) - \sum_{i=1}^{N} (r^E_i - l^E_i).
  \end{align*}
  We compute with help of Lemma~\ref{lem:diff-prod}:
   \begin{align*}
&
\sum_{\substack{c\in \col{\Gamma}\\ c\to c' }} 
\kup{\Gamma, c}_{N|M} 
-
\sum_{\substack{c''\in \col{\Phi}\\ c''\to c' }} 
\kup{\Phi, c''}_{N|M} \\
 &=\left(q^{r^E - l^{E}}\sum_{h=1}^M ([l^S_h+1][r_h ^S] -[l^S_h][r^S_h+1])q^{\sum_{i=1}^{h-1} (l^S_i - r^S_i) - \sum_{i=h+1}^{M} (l^S_i - r^S_i)} \right. \\
&\left. \hspace{0.5cm} + q^{l^S - r^{S}}\sum_{h=1}^N ([1-l^E_h][r_h ^E] -[l^E_h][1-r^E_h])q^{\sum_{i=1}^{h-1} (r^E_i - l^E_i) - \sum_{i=h+1}^{N} (r^E_i - l^E_i)} \right) 
\kup{\Upsilon, c'}_{N|M}
\\
& =\left(q^{r^E - l^{E}}\sum_{h=1}^M [r_h^S- l_h^S]q^{\sum_{i=1}^{h-1} (l^S_i - r^S_i) - \sum_{i=h+1}^{M} (l^S_i - r^S_i)} \right. \\
&\left. \hspace{2cm} + q^{l^S - r^{S}}\sum_{h=1}^N [r_h^E - l_h^E]q^{\sum_{i=1}^{h-1} (r^E_i - l^E_i) - \sum_{i=h+1}^{N} (r^E_i - l^E_i)} \right) \kup{\Upsilon, c'}_{N|M} \\
&=\left( q^{r^E - l^{E}} [r^S - l^S] + q^{l^S - r^{S}}[r^E - l^E]\right)\kup{\Upsilon, c'}_{N|M}\\
&=[k] \kup{\Upsilon, c'}_{N|M}.\qedhere
\end{align*}
  \end{proof}
\begin{cor}\label{cor:gen-square}
  The combinatorial $(N|M)$-evaluation satisfies the following local identity and its mirror image:
  \[
  \kup{\squarec}= \sum_{j=\max{(0,m-n)}}^l\begin{bmatrix}l \\ k-j \end{bmatrix} \kup{\squared}.
\]
\end{cor}
\begin{proof}
  This is an easy induction using Lemma~\ref{lem:easy-square} and Corollary~\ref{cor:easy-digon2}.
\end{proof}

\begin{lem}
  \label{lem:badsquare}
    The combinatorial $(N|M)$-evaluation satisfies the following local identity:
\begin{align} \label{eq:bad-square}
  \kup{\squaremm}_{N|M} =
 [N-M-2m]\kup{\twovertmm}_{N|M} +
  \kup{\squaremmm}_{N|M}. 
 \end{align}
\end{lem}

\begin{proof}
  Let $\Gamma$, $\Upsilon$ and $\Phi$ be three MOY graphs which are identical except in a small ball $B$ where they are related by the local relation (\ref{eq:bad-square}), $\Gamma$ being on the left-hand side and $\Upsilon$ and $\Phi$ on the right-hand side ($\Upsilon$ the first and $\Phi$ the second).
  Let $c$ be a coloring of $\Gamma$. There are two possibilities: either the two rungs receive the different colors or the receive different colors.

  First, consider a coloring $c=(\underline{\Gamma^E}, \underline{\Gamma^S})$ which gives different colors to the two rungs.
   There is no coloring of $\Upsilon$ which coincides with $c$ outside $B$. On the other hand, it induces a unique coloring $c'=(\underline{\Phi^E}, \underline{\Phi^S})$ of $\Phi$ which associates different colors to the two rungs of $\Phi$. One has:
   \begin{align*}
    m(c) &= m(c'), && \quad \quad & 
    s(c) &= s(c'),\\
w_\rho(c) &= w_\rho(c') && \quad \textrm{and} & 
    w_s(c) &= w_s(c').
  \end{align*}
  Hence
  \begin{align}\label{eq:badsquare-firstcase}
\kup{\Gamma,c}_{N|M}=  (-1)^{s(c)}q^{w_\rho(c) + w_s(c)}m(c) = (-1)^{s(c')}q^{w_\rho(c') + w_s(c)}m(c')= \kup{\Phi,c'}_{N|M}.
\end{align}

Consider now the case where $c$ gives the same color to the two rungs. It induces a coloring $c'$ of $\Upsilon$. Similarly there are colorings of $\Phi$ which induce $c'$ on $\Upsilon$. Let us write $c\to c'$ to indicate that $c$ induces the coloring $c'$ on $\Upsilon$. We will show the following: 
  \begin{align*}
    \sum_{\substack{c\in \col{\Gamma} \\ c\to c'}} \kup{\Gamma, c}_{N|M}
-
    \sum_{\substack{c\in \col{\Phi} \\ c''\to c'}} \kup{\Phi, c''}_{N|M}
=
[N-M-2m] \kup{\Upsilon, c'}_{N|M}
  \end{align*}
  which together with (\ref{eq:badsquare-firstcase}) implies the lemma.

Let us fix a coloring $c'= (\underline{\Upsilon^E}, \underline{\Upsilon^S})$ of $\Upsilon$. and denote by $e_l$ (\resp{}$e_r$) the vertical edge on the left (\resp{}right) of the part of $\Upsilon$ which is in $B$. We need a few extra notations. For $i \in \set{M}$ and $j \in \set{N}$, set:
  \begin{align*}
    l_i^S &:= \ell_{\Gamma_i^S}(e_l), &\quad \quad&& r_i^S &:= \ell_{\Gamma_i^S}(e_r), &\quad \quad&&
    l_j^E &:= \ell_{\Gamma_j^E}(e_l), &\quad \quad&& r_j^E &:= \ell_{\Gamma_j^E}(e_r), \\
    l^S &:= \sum_{i=1}^Ml_i^S, &\quad \quad&& r^S &:= \sum_{i=1}^Mr_i^S, &\quad \quad&&
    l^E &:= \sum_{j=1}^Nl_j^E, &\quad \quad&&
r^E &:= \sum_{j=1}^Nr_j^E.
  \end{align*}
  A coloring of $\Gamma$ inducing $c'$ is totally determined by the color it gives to the two rungs. Let $c= (\underline{\Gamma^E}, \underline{\Gamma^S})$ and denote by $\Gamma^A_h$ the sub-MOY graph of $\Gamma$ which has label $1$ on the two rungs. There are two possibilities: either $A=S$ and $h \in \set{M}$ or $A=E$ and $h \in \set{N}$.
  In the case $A=S$ and $h \in \set{M}$, one has:
   \begin{align*}
    m(c) &= m(c') [l^S_h+1][r_h^S + 1] \\
    s(c) &= s(c') -1 ,\\
w_\rho(c) &= w_\rho(c') + (N+M -2h +1) \quad \textrm{and}\\
    w_s(c) &= w_s(c') - r^E - l^E - \sum_{i=1}^{h-1} (l^S_i + r^S_i) + \sum_{i=h+1}^{M} (l^S_i + r^S_i).
  \end{align*}
  In the case $A=E$ and $h \in \set{N}$, one has:
   \begin{align*}
    m(c) &= m(c') [1-l^E_h][1-r_h ^E] \\
    s(c) &= s(c'),\\
w_\rho(c) &= w_\rho(c')  + (N+M -2h +1)\quad \textrm{and}\\
    w_s(c) &= w_s(c') + l^S + r^S + \sum_{i=1}^{h-1} (r^E_i + l^E_i) - \sum_{i=1}^{N} (r^E_i + l^E_i).
  \end{align*}
  Note that in order $c$ to exist it is necessary and sufficient that $r_h^E=l_h^E=1$. 
  
   We perform the same analysis with a coloring $c'' = (\underline{\Phi^E}, \underline{\Phi^S})$ inducing $c'$ on $\Upsilon$. Denote by $\Phi^A_h$ the sub-MOY graph of $\Phi$ which has label $1$ on the two rungs. There are two possibilities: either $A=S$ and $h \in \set{M}$ or $A=E$ and $h \in \set{N}$.
  In the case $A=S$ and $h \in \set{M}$, one has:
     \begin{align*}
    m(c'') &= m(c') [l^S_h][r_h ^S] \\
    s(c'') &\equiv s(c') + 1,\\
w_\rho(c'') &= w_\rho(c') + (N+M-2h+1) \quad \textrm{and}\\
    w_s(c'') &= w_s(c') - r^E - l^E - \sum_{i=1}^{h-1} (l^S_i + r^S_i)+ \sum_{i=h+1}^{M} (l^S_i + r^S_i).
  \end{align*}
Note that in order $c''$ to exist it is necessary and sufficient for $r_h^S$ and $l_h^S$ to be positive. 
  In the case $A=E$ and $h \in \set{N}$, we have:
   \begin{align*}
    m(c'') &= m(c') [l^E_h][r_h ^E] \\
    s(c'') &\equiv  s(c') ,\\
w_\rho(c'') &= w_\rho(c') + (N+M-2h+1)\quad \textrm{and}\\
    w_s(c'') &= w_s(c') + l^S + r^S + \sum_{i=1}^{h-1} (r^E_i + l^E_i) - \sum_{i=1}^{N} (r^E_i + l^E_i).
  \end{align*}
  Hence, one has:
   \begin{align*}
&    \sum_{\substack{c\in \col{\Gamma} \\ c\to c'}} \kup{\Gamma, c}_{N|M}
-
    \sum_{\substack{c\in \col{\Phi} \\ c''\to c'}} \kup{\Phi, c''}_{N|M}
\\
=& \left(\sum_{i=1}^M \left([l^S_h][r_h ^S] - [l^S_h+1][r_h ^S+1] \right) 
q^{N+M-2h +1 - r^E - l^E - \sum_{i=1}^{h-1} (l^S_i + r^S_i) + \sum_{i=h+1}^{M} (l^S_i + r^S_i)} \right. \\
&+ \left. \sum_{i=1}^N \left([1-l^E_h][1-r_h ^E] - [l^E_h][r_h ^E]  \right) 
q^{N+M-2h +1 + r^S + l^S + \sum_{i=1}^{h-1} (l^E_i + r^E_i) - \sum_{i=h+1}^{N} (l^E_i + r^E_i)} \right) \kup{\Upsilon, c'}_{N|M}\\
=&\left(\sum_{i=1}^M \left(-[l^S_h+r_h^S+1] \right) 
q^{N+M-2h +1 - r^E - l^E - \sum_{i=1}^{h-1} (l^S_i + r^S_i) + \sum_{i=h+1}^{M} (l^S_i + r^S_i)} \right. \\
&+ \left. \sum_{i=1}^N \left([1-l^E_h-r_h ^E]\right) 
q^{N+M-2h +1 + r^S + l^S + \sum_{i=1}^{h-1} (l^E_i + r^E_i) - \sum_{i=h+1}^{N} (l^E_i + r^E_i)} \right) \kup{\Upsilon, c'}_{N|M}\\
=& \left(\sum_{i=1}^M \left(-[l^S_h+r_h^S+1] \right) 
q^{N - r^E - l^E - \sum_{i=1}^{h-1} (l^S_i + r^S_i+1) + \sum_{i=h+1}^{M} (l^S_i + r^S_i+1)} \right. \\
&+ \left. \sum_{i=1}^N \left([1-l^E_h-r_h ^E]\right) 
q^{M + r^S + l^S + \sum_{i=1}^{h-1} (l^E_i + r^E_i-1) - \sum_{i=h+1}^{N} (l^E_i + r^E_i-1)} \right) \kup{\Upsilon, c'}_{N|M}\\
=& \left(-[M+ l^S + r^S] q^{N- l^E - r^E} + [N - l^E - r^E]q^{M + l^S +r^S} \right)  \kup{\Upsilon, c'}_{N|M}\\
=& [N-M - 2m] \kup{\Upsilon, c'}_{N|M}.\qedhere
\end{align*}
\end{proof}

\section{Link invariants}
\label{sec:link-invariants}
The aim of this section is to define link invariants using the combinatorics worked out in the previous section. The definition are really close from \cite{MR1659228}. The main point here is to fix normalization.

\begin{dfn}
  \label{dfn:colored-diag}
  A \emph{labeled} link diagram is an oriented link diagram whose components are labeled by non-negative integers. If $D$ is a labeled link diagram, denote by $\Xing(D)$ (or simply $\Xing$) the set of crossings of $D$. A diagram is \emph{unlabeled} or \emph{trivially labeled} if all its components are labeled by $1$. 
  For each crossing $x$ in $\Xing$, define $k_x$ and $e_x$ by the following formula:
  \begin{align*}
k_x &= 
\begin{cases}
  m(N-M-m +1) & \textrm{if $x$ is positive and the two strands have label $m$,} \\
  -m(N-M-m +1) & \textrm{if $x$ is negative and the two strands have label $m$,} \\
  0 & \textrm{else;}
\end{cases} \\
e_x &=
\begin{cases}
  m & \textrm{if the two strands of $x$ have label $m$,} \\
  0 & \textrm{else.}
\end{cases} 
\end{align*}
Finally, define $k(D)$ (\resp{}$e(D)$) to be the sum of the $k_x$ (\resp{}$e_x$) for all $x$ in $\Xing(D)$.
\end{dfn}

\begin{dfn}
  Let $D$ be a labeled link diagram. The $(N|M)$-evaluation of $D$ is the Laurent polynomial in $q$ with integral coefficients denoted by $\kup{D}_{N|M}$ defined by the two following local relations:
\begin{align}
\label{eq:extcrossplus}\kup{\scriptstyle{\NB{\tikz[scale=0.6]{\begin{scope}
  \draw[->] (1, -1) -- (-1, 1) node[near end, left] {\tiny{${m}$}};
  \fill[white] (0,0) circle (2mm);
  \draw[->] (-1, -1) -- (1, 1) node[near end, right] {\tiny{${n}$}};
\end{scope}}}}}_{N|M} &= \sum_{k= \max(0, m-n)}^m (-1)^{m-k}q^{k-m}\kup{\!\!\NB{\tikz[scale=0.9]{\begin{scope}
\coordinate (A) at (-1,-1);
\coordinate (B) at (1,-1);
\coordinate (C) at (1,1);
\coordinate (D) at (-1,1);
\coordinate (a) at (-.5,-.5);
\coordinate (b) at (.5,-.5);
\coordinate (c) at (.5,.5);
\coordinate (d) at (-.5,.5);
\draw[->] (A) -- (a) node[at start, below] {\tiny{$n$}};
\draw[->] (c) -- (C) node[at end, above ] {\tiny{$n$}};
\draw[->] (B) -- (b) node[at start , below ] {\tiny{$m$}};
\draw[->] (d) -- (D) node[at end, above] {\tiny{$m$}};
\draw[<-] (c) -- (d) node[midway, above] {\tiny{$n+k-m$}};
\draw[<-] (a) -- (b) node[midway, below] {\tiny{$k$}};
\draw[->] (a) -- (d) node[midway, left] {\tiny{$n+k$}};
\draw[->] (b) -- (c) node[midway, right] {\tiny{$m-k$}};
\end{scope}
 }}\!\!}_{N|M},\\
\label{eq:extcrossminus}
\kup{\scriptstyle{\NB{\tikz[scale=0.6]{\begin{scope}
  \draw[->] (-1, -1) -- (1, 1) node[near end, right] {\tiny{${n}$}};
  \fill[white] (0,0) circle (2mm);
  \draw[->] (1, -1) -- (-1, 1) node[near end, left] {\tiny{${m}$}};
\end{scope}}}}}_{N|M} &= \sum_{k= \max(0, m-n)}^m (-1)^{m-k}q^{m-k} \kup{\!\!\NB{\tikz[scale=0.9]{\begin{scope}
\coordinate (A) at (-1,-1);
\coordinate (B) at (1,-1);
\coordinate (C) at (1,1);
\coordinate (D) at (-1,1);
\coordinate (a) at (-.5,-.5);
\coordinate (b) at (.5,-.5);
\coordinate (c) at (.5,.5);
\coordinate (d) at (-.5,.5);
\draw[->] (A) -- (a) node[at start, below] {\tiny{$n$}};
\draw[->] (c) -- (C) node[at end, above ] {\tiny{$n$}};
\draw[->] (B) -- (b) node[at start , below ] {\tiny{$m$}};
\draw[->] (d) -- (D) node[at end, above] {\tiny{$m$}};
\draw[<-] (c) -- (d) node[midway, above] {\tiny{$n+k-m$}};
\draw[<-] (a) -- (b) node[midway, below] {\tiny{$k$}};
\draw[->] (a) -- (d) node[midway, left] {\tiny{$n+k$}};
\draw[->] (b) -- (c) node[midway, right] {\tiny{$m-k$}};
\end{scope}
 }}\!\!}_{N|M}.
\end{align}
\end{dfn}

\begin{rmk}
  In order to compute $\kup{D}_{N|M}$ one first expresses the diagram $D$ as a linear combination of MOY graphs and then uses the definition of $\kup{\bullet}_{N|M}$ for graphs given in Definition~\ref{dfn:weight}. 
\end{rmk}

\begin{thm}
  \label{thm:link-invariant}
  For any non-negative integers $M$ and $N$, the Laurent polynomial $P_{N|M}(D) :=(-1)^{e(D)}q^{k(D)}\kup{D}_{N|M}$ depends only on the oriented labeled link represented by $D$.
\end{thm}

\begin{proof}
  It is enough to check invariance under Reidemeister moves. This follows from the various identities satisfied by the $(N|M)$-evaluation given in Section~\ref{sec:colorings-moy-graphs}. The way to deduce invariance from these identities is given in~\cite{MR1659228}.
\end{proof}

\begin{rmk}
  From Corollary~\ref{cor:depends-N-M}, one deduces that for any link $L$, the polynomial $P_{N|M}(L)$ depends only on $L$ and $N-M$.
\end{rmk}

\begin{prop}
  The polynomial $P_{N|M}$ satisfies the following skein relation:
  \[
q^{M-N}    
P_{N|M}\left(
\crossposone
\right) -      
q^{N-M}    
P_{N|M}\left(
\crossnegone
\right) =      
(q^{-1}-q)
P_{N|M}\left(
\twovertsmoothone
\right).      
\]
\end{prop}

\begin{proof}
  This follows from the local definition of $\kup{\bullet}_{N|M}$ on unlabeled crossings: simpler
    \begin{align*}
      \kup{\crossposone}_{N|M}  &= \kup{\dumbleone}_{N|M} -q^{-1}\kup{\twovertsmoothone}_{N|M}, \\
      \kup{\crossnegone}_{N|M}  &=  \kup{\dumbleone}_{N|M}  - q\kup{\twovertsmoothone}_{N|M} .
    \end{align*}
  Taking in account the contribution of $k(D)$ and $(-1)^{e(D)}$ in the definition of $P_{N|M}$ gives the identity.
\end{proof}

\begin{rmk}
  The connection between $U_q(\gll_{N|M})$-link invariants for links labeled by an arbitrary partition $\lambda$ and generic versions of them known as colored HOMFLY-PT polynomials was explored by Queffelec and Sartori \cite{2015arXiv150603329Q}. We refer the reader to this paper for a detailed account. In particular, they show 
  that such invariants depend only on $N-M$ and the partition $\lambda$. The present paper provides an alternative proof for this well-known fact for partitions which are rows.
\end{rmk}

\section{Non semi-simple invariants}
\label{sec:non-semi-simple}

In this part we suppose that $1\leq M\leq N$. Let us consider a link $L$ and suppose that one of its component has label $n$ with $n >N-M$. One can show that in this case the $\gll_{N|M}$-invariant is equal to $0$ (see for instance \cite{2015arXiv150603329Q}. This follows from the fact that the $\gll_{N|M}$-invariant of the unknot labeled $n$ is equal to $\left[\begin{smallmatrix} M-N \\n\end{smallmatrix}\right]$ which happens to be $0$. The aim of this part is to  directly re-normalize the $\gll_{N|M}$ invariant in this case in order to get a non-trivial invariant. The theory of renormalized invariants is nowadays well developed but we refer to one the early paper by Geer and Patureau-Mirand which treats the case we are looking at \cite{MR2640994} and also to the treatment by Queffelec and Sartori \cite{2015arXiv150603329Q}. The interested reader could also consult the paper by Geer, Patureau-Mirand and Turaev \cite{MR2480500}.

Here, we focus on the case $M=1$ and $n=N$. Incidentally, the re-normalized $\gll_{1|1}$-invariant for links uniformly colored by $1$ equals the (one-variable) Alexander polynomial.  Hence, this construction can be thought of a generalization of the Alexander polynomial. Let uspoint out that these invariants do not coincide with the colored Alexander polynomials (or ADO invariants) \cite{MR1164114}.

\begin{dfn}
  \label{dfn:markedMOY}
  A \emph{marked MOY graph $\Gamma_\star$} is a MOY graph $\Gamma$ with a base point $\star$ in the interior of one of its edges. If the marked point is on an edge of label $k$, we say that $\Gamma_\star$ has \emph{type $k$.} 
\end{dfn}

\begin{dfn}
  \label{dfn:eval-marked}
  Let $\Gamma_\star$ be a marked MOY graph and denote $e$ the edge of $\Gamma$ which contains the base point. A \emph{$(N|1)$-coloring} of $\Gamma_\star$ is a $(N|1)$-coloring $(\underline{\Gamma^E},\underline{\Gamma^S})$ of the underlying $\Gamma$ such that for all $j$ in $\set{N}$, $\ell_{\Gamma^E_j}(e)=0$. The set of $(N|1)$-coloring of $\Gamma_\star$ is denoted by $\col[N|1]{\Gamma_\star}$ (or simply $\col{\Gamma_\star}$). 

   If $\Gamma_\star$ is a marked MOY graph and $c$ is a $(N|1)$-coloring of $\Gamma_\star$,  define $\kup{\Gamma_\star, c}_{N|1}:= \kup{\Gamma, c}_{N|1}$, and:
\[
\kup{\Gamma_\star}_{N|1} = \sum_{c\in \col[N|1]{\Gamma_\star}} \kup{\Gamma_\star,c}_{N|1}. 
\]
\end{dfn}

\begin{rmk}
  \begin{enumerate}
  \item In the previous definition, $\underline{\Gamma^S}$ contains
    exactly one sub-MOY graph: $\Gamma^S_1$. One has necessarily
    $\ell_{\Gamma^S_1}(e) = \ell_{\Gamma}(e)$.
  \item Suppose that $\Gamma_\star$ is a marked circle of label $k$ with $k\geq N$, and denote by $\Gamma$ the underlying not marked circle. One has:
    \[\kup{\Gamma_\star}_{N|1} = 1 \quad \textrm{and} \quad \kup{\Gamma}_{N|1} = 0.\]
  \end{enumerate}
\end{rmk}

The results of the previous section extend naturally:

\begin{prop}
  \label{prop:marked-far-skein-relations}
  The local relations given in Proposition~\ref{prop:rel-kups} are still valid for the evaluation of marked MOY graphs (provided the marked point is not in the ball where the relations take place).
\end{prop}

\begin{dfn} 
  \label{dfn:1n-invariant} 
  Let $\beta$ be a braid diagram, define $\kup{\beta_\star}_{N|1}$ by the following procedure: 
  \begin{enumerate}
  \item Label all strands of $\beta$ by $N$.
  \item Place a marked point on the left-most strand at the bottom of $\beta$.
  \item Close $\beta$ on the right. 
  \item Use formulas (\ref{eq:extcrossplus}) and (\ref{eq:extcrossminus}) to get rid of crossings.
  \item Evaluate the obtained marked MOY graphs with $\kup{\bullet}_{N|1}$ as given in Definition~\ref{dfn:eval-marked}.
  \end{enumerate}
    As in Section~\ref{sec:link-invariants}, $\kup{\bullet}_{N|1}$ needs to be normalized to obtain a link invariant. Define $Q_{N|1}(\beta) = (-1)^{N|\beta|}\kup{\bullet}_{N|1}$, where $|\beta|$ is the number of crossings of $\beta$.

\end{dfn}

\begin{rmk}
  Besides the appearances, the normalization used for $Q_{N|1}$ is the same as for $P_{N|1}$ (see Definition~\ref{dfn:colored-diag}) but is especially simple for the present choice of labels.
\end{rmk}

\begin{prop}
  \label{prop:invariance-1n}
  The Laurent polynomial $Q_{N|1}(\beta)$ depends only on the link represented by $\beta$. 
\end{prop}

\begin{notation} \label{not:1n-link}
  If a link $L$ can be obtained by closing a braid $\beta$, set $Q_{N|1}(L):= Q_{N|1}(\beta)$.
  
\end{notation}

\begin{rmk} 
The invariant $Q_{N|1}$ vanishes on split links.
\end{rmk}

\begin{proof}[Proof of Proposition~\ref{prop:invariance-1n}]
  We need to prove invariance under 
  \begin{itemize}
  \item Braid relations.
  \item Markov moves.
  \end{itemize}
Invariance under braid relations and stabilization follows from the general setting see \cite{MR1659228}. Invariance under conjugation by $\sigma_i$ for $i\geq 2$ is obvious. Hence the only thing to show is invariance under conjugation by $\sigma_1$. 
On the closure of $\beta$ it only changes the base point in the following way:
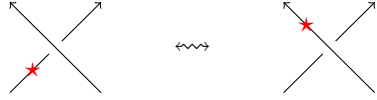
\[
\NB{\tikz[scale=0.6]{\begin{scope}
  \begin{scope}
    \draw[->] (-1,-1) -- (+1,+1) node[pos = 0.25, red] {${\star}$}; 
    \fill[white] (0,0) circle (2mm);
    \draw[->] (+1,-1) -- (-1,+1);
  \end{scope}
  \node at (3, 0) {$\leftrightsquigarrow$};
  \begin{scope}[xshift = 6cm]
    \draw[->] (-1,-1) -- (+1,+1);
    \fill[white] (0,0) circle (2mm);
    \draw[->] (+1,-1) -- (-1,+1)  node[pos = 0.75, red] {${\star}$}; 
  \end{scope}
\end{scope} }}
\]
It is enough to show that 
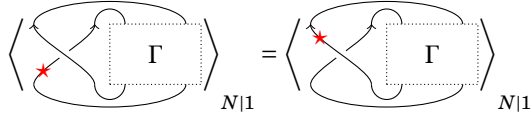
\[
\kup{\NB{\tikz[scale=0.8]{\begin{scope}[scale=0.5]
  \begin{scope}
    \draw[->] (-1,-1) .. controls +(0, 0.5) and +(0, -0.5) .. (+1,+1) node[pos = 0.25, red] {${\star}$}; 
    \fill[white] (0,0) circle (2mm);
    \draw[->] (+1,-1) .. controls +(0, 0.5) and +(0, -0.5) ..  (-1,+1);
    \draw [densely dotted] (1.5,-1) -- (4.5, -1) -- (4.5, +1) -- (1.5, 1) --cycle;
    \node at (3,0) {$\Gamma$};
    \draw (-1,  1) arc ( 180:0:2.5 and 0.75);
    \draw ( 1,  1) arc ( 180:0:0.5);
    \draw (-1, -1) arc (-180:0:2.5 and 0.75);
    \draw ( 1, -1) arc (-180:0:0.5);
  \end{scope}
\end{scope} }}}_{N|1}=
\kup{\NB{\tikz[scale=0.8]{\begin{scope}[scale=0.5]
  \begin{scope}
    \draw[->] (-1,-1) .. controls +(0, 0.5) and +(0, -0.5) .. (+1,+1);
    \fill[white] (0,0) circle (2mm);
    \draw[->] (+1,-1) .. controls +(0, 0.5) and +(0, -0.5) ..  (-1,+1)  node[pos = 0.75, red] {${\star}$}; 
    \draw [densely dotted] (1.5,-1) -- (4.5, -1) -- (4.5, +1) -- (1.5, 1) --cycle;
    \node at (3,0) {$\Gamma$};
    \draw (-1,  1) arc ( 180:0:2.5 and 0.75);
    \draw ( 1,  1) arc ( 180:0:0.5);
    \draw (-1, -1) arc (-180:0:2.5 and 0.75);
    \draw ( 1, -1) arc (-180:0:0.5);
  \end{scope}
\end{scope} }}}_{N|1}
\]
for any MOY graph $\Gamma$ with four ends. 
One can actually suppose (see \cite[Section 2.3.1]{2017arXiv171103333S}) that $\Gamma$ is equal to a linear combination of MOY graph of the form
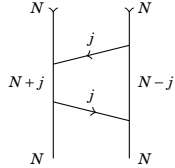
\[
\NB{\tikz[scale=1]{\begin{scope}[font= \tiny]
  \draw[-<] (0, -2) -- +(0,2) node[pos= 0.5,left] {$N+j$} node[pos= 0,left] {$N$} node[pos= 1,left] {$N$} ; 
  \draw[-<] (1, -2) -- +(0,2) node[pos= 0.5,right] {$N-j$} node[pos= 0,right] {$N$} node[pos= 1,right] {$N$}; 
  \draw[-<-] (1,-1.5) -- +(-1, 0.25) node[midway, above] {$j$};
  \draw[-<-] (0,-0.75) -- +(1, 0.25) node[midway, above] {$j$};
\end{scope} }}
\]
with $0\leq j\leq N$.
Using relation (\ref{eq:extcrossplus}), one can express $\kup{\beta_1}$ and $\kup{\beta_2}$ as:
\begin{align*}
  \kup{\beta_1}_{N|1} &= \sum_{i=0}^n \sum_{j=0}^n \lambda_{ij} \kup{\Gamma_\star^{ij}}_{N|1} 
  \quad \textrm{and} \\
  \kup{\beta_2}_{N|1} &= \sum_{i=0}^n \sum_{j=0}^n \lambda_{ij} \kup{\Phi_\star^{ij}}_{N|1} 
\end{align*}
Where
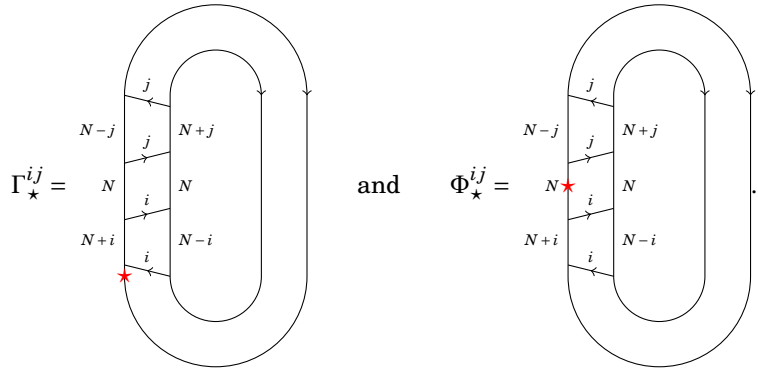
\[
  \Gamma_{\star}^{ij} = \NB{\tikz[scale= 0.6]{\begin{scope}[font= \tiny]
  \draw[->-] (0, -2) -- +(0,4) node[pos= 0.2,left] {$N+i$} node[pos= 0.5,left] {$N$} node[pos= 0.8,left] {$N-j$}  arc(180:0:2) -- +(0, -4) arc (0:-180:2);
  \draw[->-] (1, -2) -- +(0,4) node[pos= 0.2,right] {$N-i$} node[pos= 0.5,right] {$N$} node[pos= 0.8,right] {$N+j$} arc(180:0:1) -- +(0, -4) arc (0:-180:1);
  \draw[->-] (1,-2) -- +(-1, 0.25) node[midway, above] {$i$};
  \draw[->-] (0,-0.75) -- +(1, 0.25) node[midway, above] {$i$};
  \draw[->-] (0, 0.5) -- +(1, 0.25) node[midway, above] {$j$};
  \draw[->-] (1, 1.75) -- +(-1, 0.25) node[midway, above] {$j$};
  \node[red, font=\normalsize] at (0, -2) {$\star$};
\end{scope} }}
  \qquad \text{and} \qquad
  \Phi_{\star}^{ij} = \NB{\tikz[scale= 0.6]{\begin{scope}[font= \tiny]
  \draw[->-] (0, -2) -- +(0,4) node[pos= 0.2,left] {$N+i$} node[pos= 0.5,left] {$N$} node[pos= 0.8,left] {$N-j$}  arc(180:0:2) -- +(0, -4) arc (0:-180:2);
  \draw[->-] (1, -2) -- +(0,4) node[pos= 0.2,right] {$N-i$} node[pos= 0.5,right] {$N$} node[pos= 0.8,right] {$N+j$} arc(180:0:1) -- +(0, -4) arc (0:-180:1);
  \draw[->-] (1,-2) -- +(-1, 0.25) node[midway, above] {$i$};
  \draw[->-] (0,-0.75) -- +(1, 0.25) node[midway, above] {$i$};
  \draw[->-] (0, 0.5) -- +(1, 0.25) node[midway, above] {$j$};
  \draw[->-] (1, 1.75) -- +(-1, 0.25) node[midway, above] {$j$};
  \node[red, font=\normalsize] at (0, 0) {$\star$};
\end{scope} }}.
\]
Using the skein relations, one can compute $\kup{\Gamma^{ij}_\star}_{N|1}$ and $\kup{\Phi_\star^{ij}}_{N|1}$: 
\begin{align*}
  \kup{\Gamma_\star^{ij}}_{N|1} &=
  \begin{bmatrix}
    i+j \\ i
  \end{bmatrix}
  \begin{bmatrix}
    N-1 -(i+j) \\ N-i
  \end{bmatrix}
  \begin{bmatrix}
    N-1 -N \\ i
  \end{bmatrix}
  \begin{bmatrix}
    N \\ j
  \end{bmatrix}, \\
    \kup{\Phi_\star^{ij}}_{N|1} &= 
  \begin{bmatrix}
    i+j \\ i
  \end{bmatrix}
  \begin{bmatrix}
    N-1 -(i+j) \\ N-j
  \end{bmatrix}
  \begin{bmatrix}
    N-1 -N \\ j
  \end{bmatrix}
  \begin{bmatrix}
    N \\ i
  \end{bmatrix}.
\end{align*}
These two products of quantum binomials are equal.
Hence $\kup{\beta_1}_{N|1} = \kup{\beta_2}_{N|1}$ and therefore $Q(\beta_1)_{N|1}= Q(\beta_2)_{N|1}$.
\end{proof}

\begin{thm}
  \label{thm:kashaev} 
  For any $N \in \ZZ_{\geq 0}$ and any link $L$, one has:
  \[
  Q_{N|1}(L,q= e^{\frac{i\pi}{N+1}}) = J'_N(\overline{L}, q=e^{\frac{i\pi}{N+1}}),
\]
where $\overline{L}$ denotes the mirror image of $L$.
\end{thm}

Note that the sequence $\left(J'_n(\overline{L}, q=e^{\frac{i\pi}{n+1}})\right)_{n \in \ZZ_{>0}}$ appears in the formulation of the volume conjecture \cite{MR1341338, MR1828373}. For proving this theorem we need a technical result about quantum binomials evaluated at root of unity.

\begin{lem}\label{lem:qbin-root}
  Let $N$, $a$ and $b$ be three integers. The following identity holds:
  \[
    \begin{bmatrix}
      N-1 - a \\
      b
    \end{bmatrix}_{q= e^{\frac{i\pi}{N+1}}}= (-1)^b\begin{bmatrix}
      -2- a \\
      b
    \end{bmatrix}_{q= e^{\frac{i\pi}{N+1}}}
    \]
\end{lem}
\begin{proof}
  We assume here that $q = e^{\frac{i\pi}{N+1}}$.
  \begin{align*}
     \begin{bmatrix}
      N-1 - a \\
      b
    \end{bmatrix}
    &=\left(\prod_{i=1}^b\frac{q^{N-1-a+i} - q^{-N-a+1-i}}{q^{i} -
      q^{-i}}
    \right)
    \\
    &=\left(\prod_{i=1}^b\frac{-q^{-2-a +i} + q^{-2-a-i}}{q^{i} -
      q^{-i}}
    \right)
    \\
    & = (-1)^b
    \begin{bmatrix}
      -2- a \\
      b
    \end{bmatrix}. \qedhere
  \end{align*}
\end{proof}

\begin{proof}[Proof of Theorem~\ref{thm:kashaev}]
  Let us fix a braid diagram $\beta$ representing $L$ and denote $k$ the braid index of $\beta$. Once every strand of $\beta$ is labeled by $N$, the level of $\beta$ is $kN$.
  We will compare $\kup{\beta_\star}_{N|1}$ and $\frac{\kup{\beta}_{0|2}}{[-N]}$ evaluated at $q= e^{\frac{i\pi}{N+1}}$.
  The expansion of crossings is the same in both cases. Hence it is enough to compare
  \[
    \left.\kup{\Gamma_\star}_{N|1}\right|_{q= e^{\frac{i\pi}{N+1}}}
    \quad \text{and} \quad \left.\frac{\kup{\Gamma}_{0|2}}{[-N]}\right|_{q= e^{\frac{i\pi}{N+1}}}\]
  for $\Gamma_\star$ a marked MOY graph appearing in the expansion of $\beta_\star$.
  Both these quantities can be computed using the relations in Proposition~\ref{prop:rel-kups} for $(N|1)$ and $(0|2)$  
  and the fact that $\kup{U_\star}_{N|1}= \frac{\kup{U}_{0|2}}{[-N]} =1$ for $U$ a circle of label $N$.
  More precisely, one has:
  \begin{align*}
    \kup{\Gamma_\star}_{N|1} &= R_\Gamma(N-1, q)\kup{U_\star}_{N|1} \quad \text{and} \quad
    \frac{\kup{\Gamma}_{0|2}}{[-N]} = R_\Gamma(-2,q) \frac{\kup{U}_{0|2}}{[-N]},
  \end{align*}
  for a $R_\Gamma(n, q)$ a sum of product of quantum binomials of the form $\begin{bmatrix} a_i \\
      b_i
    \end{bmatrix}$ and $\begin{bmatrix}
      n- c_j \\
      d_j
    \end{bmatrix}$. The first binomials correspond to relations (\ref{eq:extrelbin1}), (\ref{eq:extrelsquare2}) and (\ref{eq:extrelsquare3}) which preserve the level, the second correspond to relations~(\ref{eq:extrelcircle}) (\ref{eq:extrelbin2}) (\ref{eq:extrelsquare1}) which do not preserve the level. Note that a binomial $\begin{bmatrix}n- c_j \\ d_j \end{bmatrix}$ appears exactly when the level decreases by $d_j$. Since to go from $\Gamma$ to $U$, the level is decreases by $(k-1)N$, for each product of binomials in $R_\Gamma(n, q)$, the sum of the $d_j$'s equals $(k-1)N$. This gives, in view of Lemma~\ref{lem:qbin-root}:
    \[ R_\Gamma\left(N-1,e^{\frac{i\pi}{N+1}}\right)=(-1)^{(k-1)N}R_\Gamma\left(-2, e^{\frac{i\pi}{N+1}}\right)
    \]
    and therefore
    \[
      \kup{\beta_\star, q= e^{\frac{i\pi}{N+1}}}_{N|1} = (-1)^{(k-1)N}\frac{\kup{\beta, q= e^{\frac{i\pi}{N+1}}}_{0|2}}{[-N]_{q= e^{\frac{i\pi}{N+1}}}}.
\]    
Since all strands are labeled by $N$, one has: 
\begin{align*}
  P_{0|2}(L) &= (-1)^{N(c_+ + c_-)} q^{N(N+1)(c_- - c_+)} \kup{\beta}_{0|2} \quad \text{and} \\
  Q_{N|1}(L) &= (-1)^{N(c_+ + c_-)}  \kup{\beta_\star}_{N|1},
\end{align*}
where $c_+$ and $c_-$ are the number of positive and negative crossings of $\beta$ respectively.
Hence, for $q = \exp(\frac{i\pi}{N+1})$, \[Q_{N|1}(L)
  =
  (-1)^{N(k-1 +c_- -c_+)} \frac{P_{0|2}(L)}{[-N]}
  =
  (-1)^{N(k +c_- -c_+)} \frac{P_{0|2}(L)}{[N]}
  .\] It turns out that $k +c_--c_+$ has the same parity as the number $\ell$ of components of $L$. This gives:
\[
Q_{N|1}(L,q= e^{\frac{i\pi}{N+1}})=\left.(-1)^{N\ell} \frac{P_{0|2}(L)}{[N]}\right|_{q= e^{\frac{i\pi}{N+1}}} = J'_N(\overline{L},q= e^{\frac{i\pi}{N+1}}).\qedhere
\]
\end{proof}
\begin{rmk}
The previous result can also be obtained in an indirect way from the work of Geer and Patureau-Mirand \cite{MR2468374}. In this paper, they prove that the generalized Links--Gould invariants specialize to the Kashaev invariants. These invariants are constructed using typical representations $V(\alpha)$ of $U_q(\gll_{N|1})$ depending on a complex parameter $\alpha$. For $\alpha=-1$, these are exactly the representations considered in the present paper and this specialization $\alpha=-1$ is compatible with the ones providing the Kashaev invariants.
\end{rmk}

\appendix

\section{MOY graphs, an algebraic approach}
\label{sec:moy-graphs-an}

For a general introduction to  the Reshetikhin--Turaev functors, we refer to Turaev's book \cite{MR1292673}. For more details on the super setting, we refer to Geer and Patureau-Mirand \cite{MR2640994} and  references therein. The later paper can also be consulted for the renormalization process (see also \cite{2015arXiv150603329Q}). Here, we follow  and expand the diagrammatic presentation of Tubbenhauer--Vaz--Wedrich \cite{tubbenhauer2015super}. For the reader convenience, we made explicit the various pieces of the Reshetikhin--Turaev functor. All proofs are omitted since they are all direct verifications.

In what follows, a \emph{super} algebra is presented. It means that it is endowed with a $\ZZ/2$-grading and more importantly, its modules are objects of $\svect$ the category of super vector spaces. Objects of this category are $\ZZ/2$-graded vector spaces, morphisms are linear maps preserving the $\ZZ/2$ grading.
The $\ZZ/2$-grading, called \emph{parity}, is denoted $|\!\bullet\!|$. Homogeneous elements with parity equal to $0$ (\resp{}$1$) are \emph{even} (\resp{}\emph{odd}).
This category inherits from $\vect$ a monoidal structure and a duality. The braiding $c$ on $\svect$ differs from that of $\vect$:
\[
\begin{array}{crcl}
  c_{V,W}  \colon\thinspace & V\otimes W & \to & W\otimes V \\
  & v\otimes w   &\mapsto & (-1)^{|v||w|}w\otimes v. 
\end{array}
\]
If $V$ is an object of $\svect$ which has finite dimension as a vector space. 
Its \emph{super dimension} $\sdim V$ is the integer $\dim V_0 - \dim V_1$, where $V= V_0 \oplus V_1$ is the $\ZZ/2$-decomposition\footnote{
The super dimension is actually the categorical dimension, the sign coming from the braiding on $\svect$.
}. The \emph{classical dimension} of $V$ is its dimension as a vector space. 
\begin{dfn} 
Let $N$ and $M$ be two non-negative integers. The \emph{quantum general linear superalgebra} $U_q(\gll_{N|M})$ is
the associative, unital, $\ZZ/2$-graded $\CC(q)$-algebra generated by $L_i$, $L_i^{-1}$, $F_j$ and $E_j$, with $1\leq i \leq N+M$ and $1\leq j \leq N+M-1$ 
subject to the nonsuper relations
\begin{gather*}
L_iL_j = L_jL_i,  \qquad L_iL_i^{-1}= L_i^{-1}L_i= 1, \\ 
L_iE_i = qE_iL_i, \qquad L_{i}E_{i-1} = q^{-1}E_{i-1}L_{i}, \qquad \text{for $i\leq N$},\\ 
L_iF_i = q^{-1}F_iL_i, \qquad L_{i}F_{i-1} = qF_{i-1}L_{i}, \qquad \text{for $i\leq N$},\\
L_iE_i = q^{-1}E_iL_i, \qquad L_{i}E_{i-1} = qE_{i-1}L_{i}, \qquad \text{for $i\geq N+1$},\\ 
L_iF_i = qF_iL_i, \qquad L_{i}F_{i-1} = q^{-1}F_{i-1}L_{i}, \qquad \text{for $i\geq N+1$},\\
L_iF_j = F_jL_i, \qquad L_iE_j = E_jL_i \qquad \text{for $j\neq i, i-1$,} \\
E_iF_j - F_jE_i = \delta_{ij}\frac{L_iL_{i+1}^{-1} - L_i^{-1}L_{i+1}}{q-q^{-1}}, \text{for $1\leq i\leq N-1$,}   \\
E_iF_j - F_jE_i = -\delta_{ij}\frac{L_iL_{i+1}^{-1} - L_i^{-1}L_{i+1}}{q-q^{-1}}, \text{for $N+1\leq i <N+M-1$,}   \\
[2]_qF_i F_j F_i = F_i^2F_j + F_j F_i^2   \qquad \text{if $|i − j| = 1$ and $i\neq N$,}  \\
[2]_qE_i E_j E_i = E_i^2 E_j + E_j E_i^2   \qquad \text{if $|i − j| = 1$ and $i\neq N$,}  \\
E_i E_j = E_j E_i, \qquad F_iF_j = F_jF_i \qquad  \text{if $|i − j| > 1$}
\end{gather*}
and the super relations
\begin{align*}
 & E_N^2= F_N^2 =0 \qquad E_NF_N + F_NE_N =\frac{L_NL_{N+1}^{-1} - L_N^{-1}L_{N+1}}{q-q^{-1}}\\
  &[2]F_{N}F_{N-1}F_{N+1}F_{N}=\\  &F_{N}F_{N-1}F_{N}F_{N+1} + F_{N}F_{N+1}F_{N}F_{N-1} + F_{N-1}F_{N}F_{N+1}F_{N}+F_{N+1}F_{N}F_{N-1}F_{N}, \\
  &[2]E_{N}E_{N-1}E_{N+1}E_{N}=\\  &E_{N}E_{N-1}E_{N}E_{N+1} + E_{N}E_{N+1}E_{N}E_{N-1} + E_{N-1}E_{N}E_{N+1}E_{N}+E_{N+1}E_{N}E_{N-1}E_{N}. 
\end{align*}
All these generators are $\ZZ/2$-homogeneous and even, but $E_N$ and $F_N$ which are $\ZZ/2$-homogeneous and odd.
\end{dfn}
\begin{prop} Defining $\Delta:U_q(\gll_{N|M})\to U_q(\gll_{N|M})^{\otimes 2}$, $S:U_q(\gll_{N|M})^{\mathrm{op}} \to  U_q(\gll_{N|M})$ and $\epsilon: U_q(\gll_{N|M})\to \CC(q)$ to be the $\CC(q)$ algebra maps defined by:
  \begin{align*}
    &\Delta(L_i^{\pm 1}) = L_i^{\pm1}\otimes L_i^{\pm1}        & \quad &S(L_i^{\pm1})= L_i^{\mp1}               & \quad & \epsilon(L_i^{\pm 1}) =1  &  \\
    &\Delta(F_i) = F_i\otimes 1 + L_i^{-1}L_{i+1}\otimes F_i  & \quad &S(F_i) = - L_iL_{i+1}^{-1}F_i           & \quad  & \epsilon(F_i)=0 & \\
    &\Delta(E_i)= E_i\otimes 1 + L_iL_{i+1}^{-1}\otimes E_i   &\quad & S(E_i) = - E_iL_i^{-1}L_{i+1}           &  \quad &\epsilon(E_i)=0&
  \end{align*}
  endows $U_q(\gll_{N|M})$ 
with a structure of $\ZZ/2$-graded Hopf algebra with antipode. Furthermore the category of finite-dimensional $U_q(\gll_{N|M})$-modules is braided.
\end{prop}
\begin{rmk}\label{rmk:swapMN}
  The Hopf algebras $U_q(\gll_{N|M})$ and $U_q(\gll_{M|N})$ are isomorphic. Indeed, one easily checks that the map
\[  \begin{array}{crcl}
\varphi  :& U_q(\gll_{N|M})   &\to     & U_q(\gll_{M|N})\\
& L_i &\mapsto &L_{M+N+1-i}^{- 1} \\
& L_i^{-1} &\mapsto &L_{M+N+1-i} \\
& E_i &\mapsto &F_{M+N-i} \\
& F_i &\mapsto &e_{M+N-i}
\end{array}
\]
induces an isomorphism of Hopf algebras.  
\end{rmk}
\begin{prop}
  Let $\CC_q^{N|M}$  be the $\ZZ/2$-graded $\CC(q)$-vector space generated by the homogeneous basis $(b_i)_{i=1,\dots, N+M}$ (with $|b_i|=0$ if $i\leq N$ and $|b_i|=1$ if $i> N$). The formulas
\begin{align*}
&L_i b_i = qb_i,& &L_i^{-1} b_{i} =q^{-1}b_{i},&  &\text{if }1\leq i\leq N, \\
&L_j b_j = q^{-1}b_j,& &L_j^{-1} b_{j} =qb_{j},&  &\text{if }N+1\leq j \leq N +M, \\
&L_i^{\pm 1} b_j = b_j && &&\text{if }i \neq j, &\\
&E_{i-1} b_i = b_{i-1}& &E_i b_{j} =0,&  &\textrm{if $i\neq j-1$,} & \\
&F_i b_i = b_{i+1}& &F_i b_{j} =0,&  &\textrm{if $i\neq j$}& 
\end{align*}
endow $\CC^{N|M}_q$ with a structure of $U_q(\gll_{N|M})$-module. 
\end{prop}

\begin{rmk}
Since $U_q(\gll_{N|M})$ is isomorphic to $U_q(\gll_{M|N})$, $\CC_q^{M|N}$ inherits a structure of $U_q(\gll_{N|M})$-module. 
Denote $\boldsymbol{1}_{\mathrm{odd}}$ the trivial\footnote{This means that the generators $L^{\pm1}_\bullet$ act by $1$ and the other generators act by $0$.} one-dimensional $U_q(\gll_{N|M})$-module concentrated in odd parity and $1_{\mathrm{odd}}$ a generator of this module.  The $U_q(\gll_{N|M})$-modules $\CC_q^{M|N}$ and $\boldsymbol{1}_{\mathrm{odd}}\otimes \CC_q^{N|M}$ are isomorphic. Indeed, one can check that

\[  \begin{array}{crcl}
\psi  :& \CC_q^{M|N}   &\to     & \CC_q^{N|M} \otimes \boldsymbol{1}_{\mathrm{odd}} \\
& b_i &\mapsto &b_{N+M-i+1}\otimes 1_{\mathrm{odd}} 
\end{array}
\]
induces an isomorphism of $U_q(\gll_{N|M})$-modules.
\end{rmk}

Denote $T\CC^{N|M}_q$ the $\ZZ/2\times \ZZ$-graded algebra generated by $\CC^{N|M}_q$ and $\mathrm{Sym_q^2}\CC^{N|M}_q$ the ideal generated by the set
\[\{
  b_i \otimes b_i | 1\leq i \leq N
  \}
  \cup
  \{
  q^{-1}b_i \otimes b_j + (-1)^{|b_i||b_j|} b_j \otimes b_i | 1\leq i<j\leq N+M
  \}.
\]
Denote $\Lambda_q \CC^{N|M}_q$ the space $T\CC^{N|M}_q\left/ \mathrm{Sym}_q^2\CC^{N|M}_q\right.$ For $k \in \ZZ_{\geq 0}$, denote $\Lambda^k_q \CC^{N|M}_q$ the $k$-degree (for the $\ZZ$-grading) part of $\Lambda_q \CC^{N|M}_q$.

\begin{prop}
  For all $k$,  $\Lambda^k_q \CC^{N|M}_q$ inherits a $U_q(\gll_{N|M})$-module structure. Its super dimension over $\CC(q)$ is $
  \begin{pmatrix}
    N-M \\k
  \end{pmatrix}$. Its classical dimension is
  \[
    \sum_{i=0}^k
    \begin{pmatrix}
    N \\k-i
  \end{pmatrix}
  \begin{pmatrix}
    M + i -1 \\ i
  \end{pmatrix}.
  \]
\end{prop}

In what follows, we define some morphisms in the category of $U_q(\gll_{N|M})$-modules. For this, we need to introduce a few notations.

The image of a pure tensor $x_1\otimes \dots \otimes x_k$ of $T \CC_q^{N|M}$ in $\Lambda^k\CC_q^{N|M}$ is denoted by
$x_1\wedge \dots \wedge x_k$. In particular one has:
\begin{align*}
  &b_i\wedge b_i =0 \quad\text{if $i\leq N$,}\\
  &b_i\wedge b_j= -(-1)^{|b_i||b_j|}qb_j\wedge b_i \quad \text{for $i<j$.}
\end{align*}

The $\CC(q)$-vector space $\Lambda^k_q\CC_q^{N|M}$ is spanned by the vectors
\[
\left(b_{i_1}\wedge  \dots \wedge b_{i_{k_1}}\wedge b_{j_1}\wedge  \dots \wedge b_{j_{k_2}} \right)
 _{\substack{
1\leq i_1<  \dots < i_{k^E}\leq N\\
N+1\leq j_1\leq  \dots \leq j_{k^S}\leq N+M\\
k = k^E + k^S
}}.
\]
If $1\leq i_1< \dots < i_{k^E}\leq N$ and
$N+1 \leq j_1 \leq \dots \leq j_{k^S} \leq N+M$, denote $b_{I,J} = b_{i_1}\wedge  \dots \wedge b_{i_{k^E}}\wedge b_{j_1}\wedge  \dots \wedge b_{j_{k^S}} $, where $I= \{i_1, \dots, i_{k^E}\}$ and $J = \{j_1, \dots, j_{k^S}\}$. Note that $I$ is a subset of $\{1, \dots, N\}$ while $J$ is a multi-subset of $\{N+1, \dots, N+M\}$. With these notations, $(b_{I,J})_{\#I + \# J = k}$ is an homogeneous basis of $\Lambda^k\CC^{N|M}_q$. Denote $(b^{I,J})_{\#I + \# J = k}$ its dual basis and
define the following morphisms\footnote{See \cite[Appendix A]{RW2} for similar definitions in the symmetric case, \ie{}$N=0$.}: 

\begin{align*}
&\begin{array}{crcl}
\Lambda_{k,\ell}  :& \Lambda_q^k V_q\otimes \Lambda_q^\ell V_q  &\to     & \Lambda^{k+\ell}_q V_q \\
  & b_{I_1, J_1} \otimes b_{I_2, J_2} & \mapsto &
  \begin{cases}
q^{-\left|(I_2\cup J_2)< (I_1 \cup J_1) \right|} b_{I_1\sqcup I_2, J_1 \sqcup J_2} &   \textrm{if $I_1\cap I_2 =\emptyset$,}    \\
0 & \textrm{else.}
  \end{cases}
\end{array}
\\
&\begin{array}{crcl}
Y_{k,\ell}  :& \Lambda^{k+\ell}_q V_q   &\to     & \Lambda_q^k V_q\otimes \Lambda_q^\ell V_q  \\
  & b_{I,J}  &\mapsto &  \displaystyle{\sum_{\substack{I_1\sqcup I_2 = I, \, J_1 \sqcup J_2 = J \\ \#I_1 + \#J_1 = k, \, \# I_2 + \# J_2 = \ell }}  [J_1, J_2]q^{\left|(I_2\cup J_2)< (I_1 \cup J_1) \right|} b_{I_1, J_1}\otimes  b_{I_2,J_2}}
\end{array} \\
&\begin{array}{crcl}
\stackrel{\leftarrow}{\cup}_{k}  :& \CC(q)   &\to     & \Lambda_q^a V_q\otimes (\Lambda_q^k V_q)^*  \\
  & 1  &\mapsto &  \displaystyle{\sum_{\#I=k}  b_{I,J}\otimes  b^{I,J}}
\end{array} \\
&\begin{array}{crcl}
\stackrel{\leftarrow}{\cap}_{k}  :& (\Lambda_q^k V_q)^*\otimes \Lambda_q^k V_q    &\to     & \CC(q) \\
  &  f\otimes x &\mapsto &  f(x)
\end{array} \\
&\begin{array}{crcl}
\stackrel{\rightarrow}{\cup}_{k}  :& \CC(q)   &\to     & (\Lambda_q^k V_q)^*\otimes \Lambda_q^k V_q  \\
  & 1  &\mapsto &  \displaystyle{\sum_{\substack{k= k^E+ k^S\\ \#I=k^E,\, \#J= k^S}}{q^{\deg[\set{M+N}]{J} - \deg[\set{M+N}]{I}}}{(-1)^{k^S}}b^{I,J}\otimes  b_{I,J}}
\end{array} \\
&\begin{array}{crcl}
\stackrel{\rightarrow}{\cap}_{k}  :& \Lambda_q^k V_q\otimes( \Lambda_q^k V_q)^*    &\to     & \CC(q) \\
&   b_{I_1, J_1} \otimes b^{I_2, J_2} &\mapsto &
\displaystyle{{q^{\deg[\set{M+N}]{I_1} - \deg[\set{M+N}]{J_1}}}
{(-1)^{|J_1|}}\delta_{I_1,I_2}\delta_{J_1, J_2}}
\end{array}
\end{align*}

We should explain what  $|I<J|$ and $[I,J]$ mean. If $I$ and $J$ are two multi-subsets of an ordered set $X$,  define
\begin{align*}
&|I<J|= \prod_{x<y \in X}I(x)J(y) \quad \text{and} \\
&[I,J]=\prod_{x\in X} \qbin{I(x)}{J(x)}.
\end{align*}

\begin{prop}
  These maps are morphisms of $U_q(\gll_{N|M})$-modules.
\end{prop}

Any MOY graph $\Gamma$ is isotopic  to a MOY graph $\Gamma'$ \emph{in good position}, this means that $\Gamma'$ can be obtained by vertically concatenating horizontal disjoint unions of the following elementary pieces:
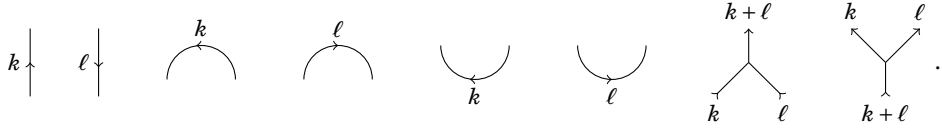
\[
\NB{\tikz[scale=0.9]{\begin{scope}[font=\small]
  \begin{scope}[xshift= -0.5cm]
    \draw[->-]  (0, -0.5) -- +(0, 1) node [midway, left] {$k$};
  \end{scope}
  \begin{scope}[xshift=0.5cm]
    \draw[-<-]  (0, -0.5) -- +(0, 1) node [midway, left] {$\ell$};
  \end{scope}
  \begin{scope}[xshift=2cm]
    \draw[-<-]  (-0.5, -0.25) arc (180:0:0.5cm) node [midway, above] {$k$};
  \end{scope}
  \begin{scope}[xshift=4cm]
    \draw[->-]  (-0.5, -0.25) arc (180:0:0.5cm) node [midway, above] {$\ell$};
  \end{scope}
  \begin{scope}[xshift=6cm]
    \draw[-<-]  (-0.5,  0.25) arc (-180:0:0.5cm) node [midway, below] {$k$};
  \end{scope}
  \begin{scope}[xshift=8cm]
    \draw[->-]  (-0.5,  0.25) arc (-180:0:0.5cm) node [midway, below] {$\ell$};
  \end{scope}
  \begin{scope}[xshift=10cm]
    \draw[->] (0,0) -- (0,0.5) node [at end, above] {$k+\ell$};  
    \draw[>-] (-0.5, -0.5) -- (0,0) node [at start, below] {$k$};  
    \draw[>-] (+0.5, -0.5) -- (0,0) node [at start, below] {$\ell$};  
  \end{scope}
  \begin{scope}[xshift = 12cm, rotate= 180]
    \draw[-<] (0,0) -- (0,0.5) node [at end, below] {$k+\ell$};  
    \draw[<-] (-0.5, -0.5) -- (0,0) node [at start, above] {$\ell$};  
    \draw[<-] (+0.5, -0.5) -- (0,0) node [at start, above] {$k$};  
  \end{scope}
\end{scope}
 }}.
\]
When a MOY graph $\Gamma'$ is in good position, interpreting these elementary pieces using the morphisms
\[
  \id_{\Lambda^k_q\CC_q^{N|M}},\,\,
  \id_{(\Lambda^\ell_q\CC_q^{N|M})^*},\,\,
  \stackrel{\leftarrow}{\cap}_{k},\,\,
  \stackrel{\rightarrow}{\cap}_{\ell},\,\,
  \stackrel{\leftarrow}{\cup}_{k},\,\,
  \stackrel{\rightarrow}{\cup}_{\ell},\,\,
  \Lambda_{k,\ell},\,\,
  Y_{k,\ell}
\] horizontal disjoint unions as tensor products and vertical concatenations as compositions, one obtains a morphism of $U_q(\gll_{N|M})$-modules $\kups{\Gamma'}_{N|M}:\CC(q) \to \CC(q)$. One sees $\kups{\Gamma'}_{N|M}$ as an element of $\CC(q)$. Using the definition of the morphisms associated with the elementary pieces, one can show that $\kups{\Gamma'}_{N|M}$ is a Laurent polynomial in $q$ with integer coefficients and symmetric in $q$ and $q^{-1}$
If $\Gamma$ is an arbitrary MOY graph, define $\kups{\Gamma}_{N|M}$ to be $\kups{\Gamma'}_{N|M}$ for $\Gamma'$ a MOY graph in good position and isotopic to $\Gamma$.
\begin{prop}\label{prop:RT-welldefined}
  The Laurent polynomial $\kups{\Gamma}_{N|M}$ is well-defined \ie it does not depend on the MOY graph in good position and isotopic to $\Gamma$ chosen to compute it.
\end{prop}

\begin{prop}\label{prop:rel-kups}
  The map $\kups{\bullet}_{N|M}: \{\text{MOY graphs} \} \to \ZZ[q, q^{-1}]$ is multiplicative under disjointe union \ie{}
  \begin{align*}
  \kups{\Gamma \sqcup \Upsilon}_{N|M}=\kups{\Gamma}_{N|M}\kups{\Upsilon}_{N|M} \quad \text{for all MOY graphs } \Gamma\mbox{ and } \Upsilon. 
  \end{align*}
   and satisfies the following the local relations and their mirror images:
  
  \begin{align}\label{eq:extrelcircle}
   \kups{\vcenter{\hbox{\tikz[scale= 0.5]{\draw[->] (0,0) arc(0:360:1cm) node[right] {\tiny{$\!k\!$}};}}}}_{N|M}=
\begin{bmatrix}
  N-M \\ k
\end{bmatrix}
\end{align}
\begin{align} \label{eq:extrelass}
   \kups{\stgamma}_{N|M} = \kups{\stgammaprime}_{N|M}
 \end{align}
\begin{align} \label{eq:extrelbin1} 
\kups{\digona}_{N|M} = \arraycolsep=2.5pt
  \begin{bmatrix}
    m+n \\ m
  \end{bmatrix}
\kups{\verta}_{N|M}
\end{align}
\begin{align} \label{eq:extrelbin2}
\arraycolsep=2.5pt
\kups{\digonb}_{N|M} = 
  \begin{bmatrix}
    N-M-m \\ n
  \end{bmatrix}
\kups{\vertb}_{N|M} 
\end{align}
\begin{align}
 \kups{\squarea}_{N|M} = \kups{\twoverta}_{N|M} + [N-M-m-1]\kups{\doubleYa}_{N|M} \label{eq:extrelsquare1}
\end{align}

\begin{align}
\kups{\squareb}_{N|M}=\!
  \begin{bmatrix}
    m-1 \\ n
  \end{bmatrix}
\kups{\bigHb}_{N|M}  +
\!\begin{bmatrix}
  m-1 \\n-1
\end{bmatrix}
\kups{\doubleYb}_{N|M} \label{eq:extrelsquare2}
\end{align}
\begin{align}
  \kups{\squarec}_{N|M}= \sum_{j=\max{(0, m-n)}}^m\begin{bmatrix}l \\ k-j \end{bmatrix}
 \kups{\squared}_{N|M}\label{eq:extrelsquare3}
\end{align}
\end{prop}

\begin{prop}[{\cite{pre06302580}}]\label{prop:completeness}
  The multiplicativity property and the local relations given in Proposition~\ref{prop:rel-kups} are enough to compute the value of $\kups{\Gamma}_{N|M}$ for any MOY graph $\Gamma$.
\end{prop}
From this statement we immediately get the following corollary which should be compared with \cite[Theorem 4.7]{2015arXiv150603329Q}.

\begin{cor}\label{cor:depends-N-M}
  The Laurent polynomial $\kups{\Gamma}_{N|M}$ only depends on $\Gamma$ and $N-M$.
\end{cor}

Two special cases are especially easy to compute: $N-M=\pm1$.

\begin{prop}
  Suppose $N-M=1$, $\kups{\Gamma}_{N|M}=0$ unless $\Gamma$ is a (maybe empty) collection of circles of label $1$ and
$\kups{\bigsqcup_i\vcenter{\hbox{\tikz[scale= 0.3]{\draw (0,0) arc(0:360:1cm) node[right] {\tiny{$\!\!1\!$}};}}}}_{N|M}=1$.
\end{prop}

\begin{prop}
  Suppose $N-M=-1$, then for any MOY graph $\Gamma$, the following identity holds:
  \[
    \kups{\Gamma}_{N|M}=  (-1)^{\rho(\Gamma)} b(\Gamma),\]
where $b(\Gamma)$ is given in Definition~\ref{dfn:sym1}.
\end{prop}

\bibliographystyle{alphaurl}
\bibliography{biblio}

\end{document}